\documentclass[final]{siamltex}
\usepackage{url}
\usepackage{graphicx,amssymb}
\usepackage{amsmath}
\usepackage{enumerate}
\usepackage{epstopdf}
\newcommand{\prox}{{\text{prox}}}
\newcommand{\sgn}{\text{sgn}}
\newcommand{\lbr}{\left(}
\newcommand{\rbr}{\right)}
\newcommand*\samethanks[1][\value{footnote}]{\footnotemark[#1]}

\title{Local and Global Convergence of an Inertial Version of Forward-Backward Splitting
\thanks{The proofs of Thms. 4.1, 5.1, 5.2, and 5.6 of this manuscript contain several errors.
These errors have been fixed in a revised and rewritten manuscript entitled ``Local and Global Convergence of a General Inertial Proximal Splitting Scheme" arxiv id. 1602.02726. We recommend reading this updated manuscript. 
}} 

\author{Patrick R. Johnstone\thanks{Beckman Institute, University of Illinois, 405 N. Mathews Ave., Urbana, IL, 61801, USA (contact: prjohns2@illinois.edu)} \and Pierre Moulin\samethanks}

\begin{document}
\maketitle
\begin{abstract}
A problem of great interest in optimization is to minimize a sum of two closed, proper, and convex functions where one is smooth and the other has a computationally inexpensive proximal operator. In this paper we analyze a family of Inertial Forward-Backward Splitting (I-FBS) algorithms for solving this problem. We first apply a global Lyapunov analysis to I-FBS and prove weak convergence of the iterates to a minimizer in a real Hilbert space. We then show that the algorithms achieve local linear convergence for ``sparse optimization", which is the important special case where the nonsmooth term is the $\ell_1$-norm. This result holds under either a restricted strong convexity  or a strict complimentary condition and we do not require the objective to be strictly convex. For certain parameter choices we determine an upper bound on the number of iterations until the iterates are confined on a manifold containing the solution set and linear convergence holds.

The local linear convergence result for sparse optimization holds for the Fast Iterative Shrinkage and Soft Thresholding Algorithm (FISTA) due to Beck and Teboulle which is a particular parameter choice for I-FBS. 
In spite of its optimal global objective function convergence rate, we show that FISTA is not optimal for sparse optimization with respect to the local convergence rate. We determine the locally optimal parameter choice for the I-FBS family. Finally we propose a method which inherits the excellent global rate of FISTA but also has excellent local rate.
\end{abstract}

\pagestyle{myheadings}
\thispagestyle{plain}
\markboth{PATRICK JOHNSTONE AND PIERRE MOULIN}{Global and Local Convergence Analysis of Inertial Forward-Backward Splitting}
\section{Introduction}

We are concerned with the following important problem:
\begin{eqnarray}
\underset{x\in\mathcal{H}}{\hbox{minimize}}\
 F(x)= f(x)+g(x),
\label{prob:1}
\end{eqnarray}
where $\mathcal{H}$ is a Hilbert space over the real numbers, the functions $f:\mathcal{H}\to\mathbb{R}\cup\{+\infty\}$ and $g:\mathcal{H}\to\mathbb{R}\cup\{+\infty\}$ are proper, convex and closed, and in addition $f$ is G\^{a}teaux differentiable, and has a Lipschitz continuous gradient. Problems of this form have come under considerable attention in recent years in applications such as machine learning \cite{koh2007interior,hastie2009elements}, compressed sensing \cite{decodebylinear05, donoho2006compressed} and image processing \cite{rudin1992nonlinear, chambolle2004algorithm} among many other examples. Of particular interest in this paper will be the special case which we will call sparse optimization (SO).
\begin{eqnarray*}
	(\hbox{Problem SO})\quad
	\underset{x\in\mathbb{R}^n}{\hbox{minimize}}\
	F(x)=
	f(x) +\rho\|x\|_1,
	\label{prob:sparse}
\end{eqnarray*} 
where $\rho\geq 0$, and $\|x\|_1=\sum_{i=1}^n |x_i|$. We refer to this problem as ``sparse optimization" because the $\ell_1$-norm encourages sparse solutions. When $f(x) = \frac{1}{2}\|b-Ax\|^2$ with $A\in\mathbb{R}^{m\times n}$ and $b\in\mathbb{R}^m$, Problem SO is often referred to as sparse least squares (Problem $\ell_1$-LS), basis pursuit denoising or LASSO. This problem is of central importance in compressed sensing and also has applications in machine learning \cite{efron2004least} and image processing \cite{vonesch2009fast}. Other important instances of Problem (\ref{prob:1}) include least squares with a total-variation \cite{osher2005iterative} or nuclear-norm \cite{CandesMC} regularizer, and minimization constrained to a closed and convex set.

\subsection{Background}
In this paper we focus on \emph{first-order splitting methods} for solving Problem (\ref{prob:1}). These methods use evaluations of $F$, gradients of the smooth part $f$, and evaluations of the \emph{proximal operator} of the nonsmooth part $g$. In particular we focus on the \emph{forward-backward splitting algorithm} (FBS), which is a classical first-order splitting approach to solving Problem (\ref{prob:1}) \cite{lions1979splitting,passty1979ergodic}. In fact FBS was developed for the more general monotone inclusion problem which includes Problem (\ref{prob:1}) as a special case. FBS involves a ``forward" step, which is an explicit gradient step with respect to the differentiable component $f$ and a ``backward" step, which is an implicit, proximal step with respect to $g$. For many popular instances of $g$ this proximal step is computationally inexpensive \cite{prox_signalProcessing}. The convergence rate of the objective function to the infimum is $O(1/k)$, which is better than the $O(1/\sqrt{k})$ rate achieved by the ``black-box" subgradient method, and is the same as if the possibly nonsmooth component were not present. Weak convergence of the iterates is also guaranteed and linear convergence occurs on strongly convex problems \cite{PolyakIntro}.  
 FBS is also commonly referred to as the proximal gradient method \cite{parikh2013proximal} and for the special case of Problem SO, it is known as the iterative shrinkage and soft thresholding algorithm (ISTA) owing to the form of the proximal step w.r.t. the $\ell_1$-norm \cite{Beck:2009,Hale:2008,bredies2008linear}. Other first-order splitting methods include ADMM \cite{eckstein2012augmented}, linearized and preconditioned ADMM \cite{ouyang2015accelerated}, primal-dual methods \cite{condat2013primal}, Bregman iterations \cite{esser2009applications} and generalized FBS \cite{raguet2013generalized}. These methods can deal with more complicated situations such as when $g$ is composed with a bounded linear operator or when the sum of $m>1$ proximable\footnote{Possessing a simple proximal operator.} functions is present. 


Nesterov developed several methods for minimizing a convex function with Lipschitz gradient (\cite{nesterov1983method}, \cite{nesterov2004introductory} chapter 2). 
These methods obtain the best objective function convergence rate possible by any first order method. Specifically, they guarantee a convergence rate of $O(1/k^2)$ for the objective function, which
is optimal in the worst case sense for convex functions with Lipschitz gradient. Note that this improves the $O(1/k)$ rate achieved by classical gradient descent.

In \cite{Beck:2009}, Beck and Teboulle extended Nesterov's method of \cite{nesterov2004introductory} to Problem (\ref{prob:1}), allowing for the presence of the possibly nonsmooth function $g$. Their method, FISTA, combines Nesterov's inertial update into an FBS framework using the same sequence of ``momentum" parameters. FISTA corresponds to a particular parameter choice for the following suite of algorithms, which we will call Inertial Forward-Backward Splitting (I-FBS). 

\begin{eqnarray*}
\text{(I-FBS)}:\forall k\in\mathbb{N},\quad
\left|
\begin{array}{ll}
y^{k+1} =& x^k + \alpha_k(x^k - x^{k-1})\\
x^{k+1} =&\prox_{\lambda_k g}\left(y^{k+1}-\lambda_k\nabla f(y^{k+1})\right) 
\end{array}
\right.
\end{eqnarray*} 
with $x^{0},x^{1}\in\mathcal{H}$ chosen arbitrarily (typically $x^0=x^1$). The sequences $\{\alpha_k\}_{k\in\mathbb{N}}$ and $\{\lambda_k\}_{k\in\mathbb{N}}$ are a subset of $\mathbb{R}_{\geq 0}$. The proximal operator $\prox_{g}:\mathcal{H}\to\mathcal{H}$ will be properly defined in Section \ref{sec:prox}. Beck and Teboulle showed that for a specific choice of $\{\alpha_k\}_{k\in\mathbb{N}}$ and $\{\lambda_k\}_{k\in\mathbb{N}}$, I-FBS obtains the optimal $O(1/k^2)$ rate in terms of the objective function,
however they did not prove convergence of the iterates $\{x^k\}_{k\in\mathbb{N}}$ to a minimizer which is also unknown for Nesterov's method. Tseng \cite{tseng2008accelerated} showed that other choices also achieve $O(1/k^2)$ rate. Recently in  \cite{chambolle2014weak}, Chambolle and Dossal considered a very similar choice of the parameters to Beck and Teboulle which obtains $O(1/k^2)$ rate in the objective function and also weak convergence of the iterates to a minimizer. Throughout the rest of the paper we will refer to all these parameter choices for I-FBS that obtain the $O(1/k^2)$ objective function rate as \emph{FISTA-like} choices. Note that FBS corresponds to I-FBS with $\alpha_k$ set to $0$ for all $k\in\mathbb{N}$ and $\lambda_k$ in the range $(0,2/L)$ where $L$ is the Lipschitz constant of $\nabla f$. Nesterov's method of \cite{nesterov2004introductory} corresponds to I-FBS with the same parameter choice as FISTA and with $g=0$. 

One of the aims of this paper is to establish broad conditions for the convergence of the iterates of I-FBS to a minimizer of Problem (\ref{prob:1}). A generalization of the I-FBS family has been studied previously in \cite{lorenz2014inertial} in the setting of monotone operator inclusion problems. However our global analysis proves convergence for a wider range of parameter choices than was proved there. An algorithm similar to I-FBS was developed in \cite{mainge2008convergence} for the more general problem of finding a fixed-point of a nonexpansive operator. However the conditions for convergence are far more strict than those developed in this paper. To the best of our knowledge the conditions for weak convergence of the iterates of I-FBS developed in this paper are novel in the literature. A more detailed comparison with existing literature is given in Section \ref{sec:fistalyap}. 

It has been observed that for the special case of Problem SO FBS exhibits \emph{local linear convergence} (see e.g. \cite{Hale:2008,bredies2008linear,liang2014local,agarwal2010fast}), elsewhere called \emph{eventual linear convergence} \cite{demanet2013eventual}.
By this it is meant that there exists some $N>0$ such that for all $k>N$ the iterates $x^k$ are confined to a manifold containing the solution set and convergence to a solution is linear. 
It is not known whether I-FBS (including the FISTA-like choices) obtains local linear convergence for Problem SO, however recently \cite{tao2015local} has made progress for the special case of Problem $\ell_1$-LS. 
In this paper, we address this by establishing local linear convergence of I-FBS for Problem SO for a broad range of parameter choices including the FISTA-like choices. Of course local linear convergence of the sequence $\{x^k\}_{k\in\mathbb{N}}$ implies convergence of the entire sequence.

\subsection{Contributions of this Paper}

In the first part of the paper, we analyze I-FBS with an appropriate multi-step Lyapunov function. This approach 
allows us to develop novel conditions on the algorithmic parameters that imply convergence of the iterates to a minimizer (weak convergence in a real Hilbert space, ordinary convergence in $\mathbb{R}^n$). This widens the range of possible parameter choices beyond those proposed in prior art such as \cite{lorenz2014inertial}. 


In the second part of the paper, we consider in detail the behavior of I-FBS applied to Problem SO. 
We show that after a finite number of iterations I-FBS reduces to minimizing a local function on a reduced support subject to an orthant constraint. This result holds for the FISTA-like choices along with a wide range of other parameter choices. Next we show that a simple ``locally optimal" parameter choice for I-FBS obtains a local linear convergence rate with the best asymptotic iteration complexity. The asymptotically optimal iteration complexity is better than that obtained by the FISTA-like choices and by ISTA. The improvement gained by I-FBS over ISTA when the correct amount of momentum is added is equivalent to the improvement that Nesterov's accelerated method \cite{nesterov2004introductory} achieves over gradient descent for strongly convex functions with Lipschitz gradients. As a corollary of our analysis, we show that the adaptive momentum restart scheme proposed in \cite{o2012adaptive} achieves the optimal iteration complexity. In conrast the analysis in \cite{o2012adaptive} is only valid for strongly convex quadratic functions. Finally for parameter choices for which the ``momentum parameter" $\alpha_k$ is bounded away from $1$, we determine an explicit upper bound on the number of iterations until convergence to the optimal manifold.

With little effort our analysis of I-FBS for Problem SO can be adapted to apply to the splitting inertial proximal method (SIPM) proposed by Moudafi and Oliny \cite{MoudafiOliny}. This method is a direct generalization of the heavy ball with friction method (HBF) \cite{polyak1964some} to proximal splitting problems and differs from I-FBS in that the gradient of $f$ is computed at $x^k$ rather than $y^{k+1}$. We show that SIPM also achieves local linear convergence for this problem under appropriate parameter constraints.    



The paper is organized as follows. In Section \ref{sec:prelim}, notation and assumptions are discussed. In Section \ref{sec:fistalyap}, we precisely define the I-FBS family  and discuss known convergence results in more detail. In Section \ref{sec:fistalyapanal} we apply our Lyapunov analysis to I-FBS. In Section \ref{sec:l1ls} we derive convergence results for Problem SO. Finally, numerical experiments are presented in Section \ref{sec:sims}. 

\section{Preliminaries}
\label{sec:prelim}
\subsection{Notation and Definitions} 
Throughout the paper, $\mathcal{H}$ is a Hilbert space over the field of real numbers, $\langle\cdot,\cdot\rangle$ is the inner product and $\|\cdot\|$ is the associated norm. Let $\Gamma_0(\mathcal{H})$ be the set of all closed, convex and proper functions whose domain is a subset of $\mathcal{H}$ and range is a subset of $\mathbb{R}\cup\{+\infty\}$. For any $g:\mathcal{H}\to\mathbb{R}\cup\{+\infty\}$ and point $x\in\mathcal{H}$, we denote by $\partial_{\epsilon} g(x)$ for $\epsilon\geq 0$ the \emph{$\epsilon$-enlargement} of the \emph{subdifferential}, defined as the set
\begin{eqnarray}
\partial_{\epsilon} g(x) \triangleq \{v\in\mathcal{H}:g(y)\geq g(x)+\langle v,y-x\rangle-\epsilon, \forall y\in\mathcal{H}\}
\label{eq:ineq3}
\end{eqnarray}
which is always convex and closed and may be empty.  We will use $\partial g$ to denote $\partial_0 g$. When $\partial g(x)$ is a singleton we will call it the \emph{gradient} at $x$, denoted by $\nabla g(x)$. 

For $a:\mathbb{R}\to\mathbb{R}$ and $b:\mathbb{R}\to\mathbb{R}$, the notation $a(k)=O(b(k))$ (resp. $a(k)=\Omega(b(k))$)  means there exists a constant $C\geq0$ such that $\lim_{k\to\infty}a(k)/b(k)\leq C$ (resp. $\lim_{k\to\infty}a(k)/b(k)\geq C$). The notation $a(k)=o(b(k))$ means $\lim_{k\to\infty}a(k)/b(k)=0$.
We will say a sequence $\{x^k\}_{k\in\mathbb{N}}\subset\mathcal{H}$ converges linearly to $x^*\in\mathcal{H}$ with rate of convergence $q\in(0,1)$, if $\|x^k-x^*\|=O(q^k)$. To be precise we will occasionally refer to this as \emph{asymptotic} or \emph{local} linear convergence. Note that this is different from nonasymptotic, or global linear covergence with rate $q$, in which case there exists a $C\geq 0$ such that $\|x^k-x^*\|\leq C q^k$ \emph{for all} $k\in\mathbb{N}$. In contrast local linear convergence allows for a finite number of iterations where such a relationship does not hold.

Define the \emph{optimal value} of Problem (\ref{prob:1}) as
\begin{eqnarray*}
F^* \triangleq \inf_{x\in\mathcal{H}} F(x)
\end{eqnarray*}
and the \emph{solution set} as
\begin{eqnarray*}
X^* \triangleq \{x\in\mathcal{H}:F(x) = F^*\}.
\end{eqnarray*}
Given a function $a:\mathbb{R}\to\mathbb{R}$, we say that the \emph{iteration complexity} of a method for minimizing $F$ is $\Omega\lbr a(\epsilon)\rbr$ if $k=\Omega\lbr a\lbr\epsilon\rbr\rbr$ implies $F(x^k)-F^*=O(\epsilon)$. To be precise we will occasionally refer to this as the asymptotic iteration complexity. 

For a matrix $A\in\mathbb{R}^{m\times n}$ and a set $S\subset\{1,2,\ldots,n\}$, $A_S$ will denote the matrix in $\mathbb{R}^{m\times |S|}$ formed by taking the columns corresponding to the elements of $S$. For a vector $v\in\mathbb{R}^{n}$, $v_S$ will denote the $|S|\times 1$ vector with entries given by the entries of $v$ on the indices corresponding to the elements of $S$, and $(v_S,0)$ will denote the vector in $\mathbb{R}^n$ equal to $v$ on the indices corresponding to $S$ and equal to zero everywhere else. Given $c\in\mathbb{R}$ and $x\in\mathbb{R}^n$, $\sgn(c)$ is defined as $+1$ if $c\geq0$ and $-1$ if $c<0$, $\sgn(x)$ is simply applying $\sgn(\cdot)$ element-wise. We will use the notation $[c]_+=\max(c,0)$. 


\subsection{Proximal Operators}
\label{sec:prox}
The proximal operator $\prox_g:\mathcal{H}\to\mathcal{H}$ w.r.t. a function $g\in\Gamma_0(\mathcal{H})$ is defined implicitly by
\begin{eqnarray*}
y-\prox_g(y) \in \partial g(\prox_g(y)),
\end{eqnarray*}
and explicitly by
\begin{eqnarray}
\prox_g(y)=\arg\min_x\left\{\frac{1}{2}\|x-y\|^2+ g(x)\right\}.
\label{eq:prox_def}
\end{eqnarray}
Since the function being minimized in (\ref{eq:prox_def}) is strongly convex and in $\Gamma_0(\mathcal{H})$, $\prox_g(y)$ exists and is unique for every $y\in\mathcal{H}$ thus it is a well defined mapping with domain equal to $\mathcal{H}$.
To be more general we will actually use the $\epsilon$-enlarged proximal operator, which is the set
\begin{eqnarray*}
\prox_g^{\epsilon}(y) = \{v:y-v\in\partial_{\epsilon}g(v)
\},
\end{eqnarray*}
which is not necessarily uniquely defined (except when $\epsilon=0$). Note that $\prox_g(y)\in\prox_g^{\epsilon}(y)$ for all $\epsilon\geq 0$. The use of $\prox_g^{\epsilon}$ allows for some approximation error in the computation of the proximal operator. 
\subsection{Cocoercivity and Convexity}  
We say that a G\^{a}teaux differentiable and convex function $f$ has a $\frac{1}{L}$-cocoercive gradient with $L> 0$, if
\begin{eqnarray}
\langle\nabla f(y)-\nabla f(x),y-x\rangle\geq \frac{1}{L}\|\nabla f(y)-\nabla f(x)\|^2,\ \forall x,y\in\mathcal{H}.
\label{eq:coercdef}
\end{eqnarray}
Note this is equivalent to the gradient being $L$-Lipschitz continuous, i.e.
\begin{eqnarray}
\|\nabla f(y)-\nabla f(x)\|\leq L\|y-x\|,\ \forall x,y\in\mathcal{H},
\label{eq:lips}
\end{eqnarray}
For a proof see \cite{abbas2014dynamical} Lemma 1.4 and the Baillon-Haddad Theorem \cite{bauschke2009baillon}. We will need the following two standard properties of such a function. For all $u,v\in\mathcal{H}$:
\begin{eqnarray}
\label{eq:ineq1}
f(u)-f(v)&\leq&\langle\nabla f(v),u-v\rangle + \frac{L}{2}\|u-v\|^2,
\end{eqnarray}
and (by convexity)
\begin{eqnarray}
\label{eq:ineq2}
f(u)-f(v)&\leq&\langle\nabla f(u),u-v\rangle.
\end{eqnarray}
We are now ready to formerly state our Assumptions for Problem (\ref{prob:1}).
\vspace{0.2cm}

{\bf Assumption 1.} 
$f$ and $g$ are in $\Gamma_0(\mathcal{H})$, $f$ is G\^{a}teaux differentiable everywhere and has a $1/L$-cocoercive gradient with $L>0$, and $F^*>-\infty$.
\vspace{0.2cm}

\subsection{Properties of Sparse Optimization}
\label{sec:propsparse}
We now outline our assumptions for Problem SO and discuss some of its properties.

{\bf Assumption SO.}
$f\in\Gamma_0(\mathcal{H})$, is \emph{twice} differentiable everywhere, and has a $1/L$-cocoercive gradient with $L>0$. 
$F^*>-\infty$ and $X^*$ is non-empty. 
\vspace{0.2cm}

 The main difference between Assumption SO and Assumption 1 is that we additionally assume that $f$ is twice differentiable. Let $H(x)$ denote the Hessian of $f$ at $x$. Then the Lipschitz constant $L$ of the gradient is equal to the supremum of the largest eigenvalue of $H(x)$ over all $x$. Furthermore note that $\|\cdot\|_1$ is in $\Gamma_0(\mathcal{H})$. Finally note that for $\rho>0$ the function $f(x)+\rho\|x\|_1$ is coercive thus $X^*$ is non-empty. 
 
 Problem SO includes Problem $\ell_1$-LS, defined as
\begin{eqnarray*}
(\hbox{Problem }\ell_1\hbox{-LS})\quad
\underset{x\in\mathbb{R}^n}{\hbox{minimize}}\
F(x)=
\frac{1}{2}\|b-Ax\|^2+\rho\|x\|_1,
\label{prob:l1ls}
\end{eqnarray*}
where $A\in\mathbb{R}^{m\times n}$ and $b\in\mathbb{R}^m$. The solution set $X^*$ of Problem $\ell_1$-LS is always non-empty. The function $f$ has gradient equal to $A^T (Ax-b)$ which is Lipschitz-continuous with Lipschitz constant $L$ equal to the largest eigenvalue of $A^T A$.

The proximal operator associated with $\rho\|\cdot\|_1$ is the \emph{shrinkage and soft-thresholding} operator $S_{\rho}(v):\mathbb{R}\to\mathbb{R}$, applied element-wise. It is defined as $S_{\rho}(v)\triangleq\left[|v|-\rho\right]_+\sgn(v)$, and thus
\begin{eqnarray}
\label{soft_thresh_def}
\{\prox_{\rho\|\cdot\|_1}(z)\}_i=S_{\rho}(z_i),\quad i=1,2,\ldots,n.
\end{eqnarray}

In the analysis of I-FBS applied to Problem SO we will need the following result proved in \cite{Hale:2008}. 
\begin{theorem}[Theorem 2.1 \cite{Hale:2008}]
For problem SO suppose Assumption SO holds, then there exists a vector $h^*\in\mathbb{R}^n$ such that for all $x^*\in X^*$, $\nabla f(x^*)=h^*$. Furthermore, for all $i\in\{1,2,\ldots,n\}$,
\begin{eqnarray*}
\frac{h_i^*}{\rho}\left\{
    \begin{array}{lr}
      =-1\hbox{ if } \exists\ x\in X^*:x_i>0 \\
      =+1 \hbox{ if } \exists\ x\in X^*:x_i<0\\
      \in [-1,1]\quad \text{else}.
    \end{array}
  \right.
\end{eqnarray*}
\label{thm:constGrad}
\end{theorem}
 The following two sets also used in \cite{Hale:2008} will also be crucial to our analysis. Let
$
D\triangleq\{i:|h_i^*|<\rho\}
$
and
$
E\triangleq \{i:|h_i^*|=\rho\}.
$
Note that $D\cap E=\emptyset$ and $D\cup E=\{1,2,\ldots,n\}$. By Theorem \ref{thm:constGrad}, we can infer that $\supp(x^*)\subseteq E$ for all $x^*\in X^*$. Finally, define
\begin{eqnarray*}
\omega\triangleq\min\{\rho-|h^*_i|:i\in D\}> 0.
\end{eqnarray*}

We will need the following Lemma proved in \cite{Hale:2008}.
\begin{lemma}[Lemma 4.1 \cite{Hale:2008}]
Under Assumption SO, if $\lambda\in[0,2/L)$,
$$
\|x-\lambda\nabla f(x) - (y-\lambda\nabla f(y))\|\leq\|x-y\|,\ \forall x,y\in\mathbb{R}^n.
$$
\label{lemma:nonexpansive}
\end{lemma} 
An alternative definition of cocoercivity is to say that if an operator $T:\mathcal{H}\to\mathcal{H}$ is $\gamma$-cocoercive than $\gamma T$ is \emph{firmly nonexpansive}. Thus Lemma \ref{lemma:nonexpansive} is just an elementary property of firmly nonexpansive operators (see Proposition 4.2 (iii), and Proposition 4.33 \cite{bauschke2011convex}). 
 
Finally, the following properties of $S_\nu$ will be useful. 
\begin{lemma}[Lemma 3.2 \cite{Hale:2008}]
Fix any $a$ and $b$ in $\mathbb{R}$, and $\nu\geq 0$:
\begin{itemize}
\item The function $S_\nu$ is nonexpansive. That is,
\begin{equation*}
|S_\nu(a)-S_\nu(b)|\leq|a-b|.
\end{equation*}
\item If $|b|\geq\nu$ and $\sgn(a) \neq \sgn(b)$ then 
\begin{eqnarray*}
|S_{\nu}(a)-S_{\nu}(b)|\leq|a-b|-\nu.
\end{eqnarray*}
\item If $S_{\nu}(a)\neq 0=S_{\nu}(b)$ then $|a|>\nu,|b|<\nu$ and 
\begin{eqnarray}
|S_{\nu}(a)-S_\nu(b)|\leq|a-b|-(\nu-|b|).
\label{eq:lemmaSv3}
\end{eqnarray}
\end{itemize}
\label{lemma:Sv}
\end{lemma}

\section{I-FBS}
\label{sec:fistalyap}
To be more general, our global analysis will apply to the following I-FBS family.
\begin{eqnarray*}
\text{(I-FBS-$\epsilon$)}:\forall k\in\mathbb{N}\quad\left|
\begin{array}{ll}
y^{k+1} =& x^k + \alpha_k(x^k - x^{k-1})\label{fists1}\\
x^{k+1} \in&\prox_{\lambda_k g}^{\epsilon_k}\left(y^{k+1}-\lambda_k\nabla f(y^{k+1})\right) 
\end{array}
\right.
\end{eqnarray*}
with $x^0,x^1\in\mathcal{H}$ chosen arbitrarily. Note that for any $\epsilon_{k}\geq0$,
\begin{eqnarray*}
\prox_{\lambda_k g}(y^{k+1}-\lambda_k\nabla f(y^{k+1}))
\in
\prox_{\lambda_k g}^{\epsilon_k}\left(y^{k+1}-\lambda_k\nabla f(y^{k+1})\right). 
\end{eqnarray*}
We will refer to $\{\alpha_k\}_{k\in\mathbb{N}}$ as the ``momentum" parameters and $\{\lambda_k\}_{k\in\mathbb{N}}$ as the ``step-size" parameters. The algorithm differs from I-FBS in that it uses the $\epsilon$-enlarged sub-differential, allowing for some error in the computation of the proximal operator. 
\subsection{Known Convergence Results}
Beck and Teboulle \cite{Beck:2009} proposed the following choice of parameters for I-FBS (I-FBS-$\epsilon$ with the $\epsilon_k$ set to $0$ for all $k\in\mathbb{N}$), 
\begin{eqnarray}
\forall k\in\mathbb{N},\,\lambda_k=\frac{1}{L},\, \alpha_k = \frac{t_k-1}{t_{k+1}},\hbox{ where } t_{k+1} = \frac{1+\sqrt{4t_k^2+1}}{2},\, t_1=1.
\label{eq:fista_mo2}
\end{eqnarray}
The method is known as FISTA. With this choice of parameters, Beck and Teboulle showed that the objective function converges to the minimum at the worst-case optimal rate of $O(1/k^2)$. In fact the $O(1/k^2)$ rate holds for a variety of choices of $\{\alpha_k\}_{k\in\mathbb{N}}$ which all have the form: $\alpha_k=1-O(1/k)$ \cite{tseng2008accelerated}. However the choice in (\ref{eq:fista_mo2}) guarantees the largest possible decrease in a given upper bound of $F(x^k)$ at each iteration.
Chambolle and Dossal \cite{chambolle2014weak} considered I-FBS with a similar choice of $\{\alpha_k\}_{k\in\mathbb{N}}$ to what was proposed by Beck and Teboulle. They investigated, for some $a>2$, 
\begin{eqnarray}
\forall k\in\mathbb{N},\,0<\lambda_k \leq \frac{1}{L},\, \alpha_k = \frac{t_k-1}{t_{k+1}},\,\hbox{ where }\ t_{k+1} = \frac{k+a-1}{a},\, t_1=1.
\label{eq:fista_mo3}
\end{eqnarray}
With this choice of parameters, the authors showed that the objective function achieves the optimal $O(1/k^2)$ convergence rate and in addition $\{x^k\}_{k\in\mathbb{N}}$ weakly converges to a minimizer. 

In contrast to \cite{Beck:2009} and \cite{chambolle2014weak}, our analysis establishes weak convergence of the iterates for a wide range of parameter choices. Indeed, the momentum sequence is not constrained to follow a particular sequence relationship, but instead must be constrained to $\alpha_k\in[0,1]$ and $\lim\sup\alpha_k<1$. However we do not guarantee the $O(1/k^2)$ objective function rate.

Lorenz and Pock \cite{lorenz2014inertial} generalized I-FBS to the problem of finding a zero of the sum of two maximal monotone operators $A$ and $B$, one of which is cocoercive. Setting $A=\nabla f$ and $B=\partial g$ recovers Problem (\ref{prob:1}). They also replaced the scalar step-size $\lambda_k$ with a general positive definite operator $\lambda_k M^{-1}$. Lorenz and Pock proved weak convergence of the iterates to a solution provided certain restrictions on $\{\alpha_k\}_{k\in\mathbb{N}}$ and $\{\lambda_k\}_{k\in\mathbb{N}}$. The restrictions on $\alpha_k$ are stronger than those derived in our global analysis. In their analysis, if the step-size $\lambda_k$ is fixed to $1/L$, $\alpha_k$ is restricted to be less than $\sqrt{5}-2\approx0.24$, whereas, as we shall see in Section \ref{sec:lyap}, our Lyapunov analysis allows $\alpha_k\in[0,1]$, so long as $\lim\sup\alpha_k<1$. For the step-size, their conditions are less restrictive than ours,  allowing for values of $\lambda_k$ up to $2/L$, whereas our analysis only allows up to $1/L$. However in their analysis larger values of $\lambda_k$ lead to a smaller range of feasible values for $\alpha_k$ reducing to $0$ as $\lambda_k$ approaches $2/L$.

In \cite{mainge2008convergence}, an inertial version of the classical Krasnosel'ski\u{i}-Mann (KM) algorithm was analyzed. The KM algorithm finds the fixed points of a nonexpansive operator $T:\mathcal{H}\to\mathcal{H}$. Setting the operator $T(x)=\prox_{\lambda_k g}(x - \lambda_k \nabla f(x))$ in the inertial KM method of \cite{mainge2008convergence} recovers I-FBS, since a point $x$ is a fixed point of $T$ if and only if it is a solution of Problem (\ref{prob:1}). The analysis of \cite{mainge2008convergence} proves weak convergence of the iterates to a fixed point but relies on verifying conditions of the form: $\sum \alpha_k\|x^k - x^{k-1}\|^p<\infty$ with $p$ equal to $1$ and $2$. In general this condition must be enforced online, restricting the range of possible choices for the sequence of momentum parameters. However, it was shown in \cite{mainge2007inertial} that choosing $\alpha_k$ to be nondecreasing and satisfying $\alpha_k\in[0,\overline{\alpha})$ with $\overline{\alpha}<1/3$ suffices to ensure the condition is satisfied and thus prove weak convergence. This condition is more restrictive than the ones derived in this paper for the special case of Problem (\ref{prob:1}).   

\subsection{Known Convergence Results for Sparse Optimization}
\label{sec:knownfistl1ls}
 The FISTA-like sequences for $\{\alpha_k\}_{k\in\mathbb{N}}$ defined in (\ref{eq:fista_mo2}) and (\ref{eq:fista_mo3}) both converge to $1$.  As we will see in Section \ref{sec:underD} this is not desirable for Problem SO. In the language of dynamical systems, when the momentum is too high the iterates move into an ``underdamped regime" leading to oscillations in the objective function and slow convergence (see \cite{o2012adaptive} for an analysis in the strongly-convex quadratic case). We will show that for Problem SO the FISTA-like choices are not optimal from the viewpoint of asymptotic rate of convergence under a local strong-convexity assumption (Corollary \ref{cor:FISTAforL1LS}) or a strict complimentarity condition (Corollary \ref{cor:strict_comp}).

In \cite{tao2015local}, the behavior of ISTA and FISTA (i.e. I-FBS with parameter choice (\ref{eq:fista_mo2})) applied to Problem $\ell_1$-LS was investigated through a spectral analysis. The authors show that both algorithms obtain local linear convergence for this problem, under the condition that the minimizer is unique, but without an estimate for the number of iterations until convergence to the optimal manifold. Furthermore they determine that the local rate of convergence of FISTA is worse than ISTA, while the transient behavior of FISTA is better than ISTA. Therefore they suggest switching from FISTA to ISTA once the optimal manifold has been identified. Our contribution differs in several ways. We note that the poor local performance of the FISTA-like choices is due to having the momentum parameter converge to $1$. Therefore we determine the optimal value for the momentum parameter that should be used in the asymptotic regime which allows for a better asymptotic rate than both ISTA and FISTA and suggest a heuristic method for estimating the optimal momentum. We also show that the adaptive restart method of \cite{o2012adaptive} will achieve the $O(1/k^2)$ rate in the transient regime and the optimal asymptotic rate. Furthermore our analysis holds for Problem SO with Problem $\ell_1$-LS as a special case and we do not require the minimizer to be unique. Finally, in the case where $\lim\sup\alpha_k<1$, we provide explicit upper bounds on the number of iterations until I-FBS has converged to the optimal manifold. 

In \cite{lin2014adaptive} a method was developed for solving Problem (\ref{prob:1}) when $f$ is strongly convex. The method is equivalent to I-FBS with the same prescription for $\{\alpha_k\}_{k\in\mathbb{N}}$ as determined by Nesterov for his method for minimizing strongly convex functions (constant scheme 2.2.8. of \cite{nesterov2004introductory}). However it also includes a backtracking procedure for adjusting $\{\lambda_k\}_{k\in\mathbb{N}}$ and $\{\alpha_k\}_{k\in\mathbb{N}}$ when the strong convexity and Lipschitz gradient parameters are not known. 
The authors of \cite{lin2014adaptive} also extended their method to Problem $\ell_1$-LS including the case where $f$ is not strongly convex. 
The authors showed that under conditions on the matrix $A$ related to the Restricted Isometry Property (RIP) used in compressed sensing, their algorithm obtains nonasymptotic (global) linear convergence, so long as the initial vector is sufficiently sparse. However, as the authors note the RIP-like conditions are much stronger than those typically found in the literature. Indeed the conditions are much stronger than those required in our proof of local linear convergence. We establish that I-FBS obtains local linear convergence regardless of the initialization point. Furthermore no RIP-like assumptions are necessary. Local linear convergence can be proved under the mild condition that the smallest eigenvalue of the Hessian restricted to the support of a minimum is non-zero at the minimum point. Or if this does not hold, under a common strict-complementarity condition (see Section \ref{sec:paramsfist}). That being said, it should be noted that local linear convergence is not as strong a statement as global linear convergence

\section{A Global Analysis of I-FBS}
\label{sec:fistalyapanal}
This section derives conditions on $\{\epsilon_k\}_{k\in\mathbb{N}}$, $\{\alpha_k\}_{k\in\mathbb{N}}$ and $\{\lambda_k\}_{k\in\mathbb{N}}$ which imply weak convergence of the iterates $\{x^k\}_{k\in\mathbb{N}}$ of I-FBS to a minimizer of Problem (\ref{prob:1}). Throughout the rest of the paper, let $\Delta_{k+1}$ denote $x^{k+1}-x^k$. Given $S,T\subset\mathcal{H}$, define $d(S,T)\triangleq\min_{s\in S,t\in T}\|s-t\|$.

\label{sec:lyap}

\begin{theorem}
Suppose that Assumption 1 holds. Assume $\{\lambda_k\}_{k\in\mathbb{N}}$ is non-decreasing and satisfies $0<\lambda_k\leq 1/L$ for all $k$, and $\{\epsilon_k\}_{k\in\mathbb{N}}$ satisfies $\epsilon_k\geq0$ for all $k$ and $\sum_k\epsilon_k<\infty$. If $0\leq\alpha_k\leq 1$ for all $k$ and $\lim\sup\alpha_k<1$, then for the iterates of I-FBS-$\epsilon$, we have
\begin{enumerate}[(i)]
\item $\sum_{k=1}^\infty \|\Delta_k\|^2 <\infty$.\label{state:finitesum}
\item $\lim_{k\to\infty}d(0,\nabla f(x^k)+\partial_{\epsilon_k} g(x^k)) = \lim_{k\to\infty}d(0,\nabla f(y^k)+\partial_{\epsilon_k} g(y^k))=0$.\label{state:convgrad}
\item If, in addition, $X^*$ is non-empty, then $x^k$ converges weakly to some $\hat{x}\in X^*$.\label{state:weak}
\end{enumerate}
\label{thm:fista_lyap}
\end{theorem}
\begin{proof}
The proof consists of two parts. In the first, we prove statements (\ref{state:finitesum}) and (\ref{state:convgrad}) using arguments inspired by Alvarez' analysis of the inertial proximal method in \cite{alvarez2000minimizing}. In the second part, we invoke Opial's Lemma \cite{opial1967weak} to prove statement (\ref{state:weak}). The second part is inspired by the analysis of the splitting inertial proximal algorithm by Moudafi and Oliny in \cite{MoudafiOliny}. 
\subsubsection*{Proof of statements (\ref{state:finitesum}) and (\ref{state:convgrad})}
Define the \emph{Lyapunov function}, or discrete energy, to be
\begin{eqnarray*}
E_{k} \triangleq \frac{\alpha_k}{2\lambda_k}\|\Delta_{k}\|^2+f(x^{k})+g(x^k).
\end{eqnarray*}
Note that this is the same energy function used by Alvarez \cite{Alvarez2000}. Inequalities (\ref{eq:ineq1}), (\ref{eq:ineq2}) and (\ref{eq:ineq3}) imply
\begin{eqnarray}
\nonumber E_{k+1}-E_k &=& \frac{\alpha_{k+1}}{2\lambda_{k+1}}\|\Delta_{k+1}\|^2-\frac{\alpha_k}{2\lambda_k}\|\Delta_k\|^2+f(x^{k+1})-f(x^k)+g(x^{k+1})-g(x^k)\\\nonumber
&\leq&\frac{\alpha_{k+1}}{2\lambda_{k+1}}\|\Delta_{k+1}\|^2
-\frac{\alpha_k}{2\lambda_k}\|\Delta_k\|^2
+\langle\nabla f(y^{k+1})+v,\Delta_{k+1}\rangle 
\\
&&+ \frac{L}{2}\|x^{k+1}-y^{k+1}\|^2+\epsilon_{k},\quad\forall\ v\in\partial_{\epsilon_{k}} g(x^{k+1}).
\label{eq:lyap1}
\end{eqnarray}
Note that the existence of a subgradient $v$ is guaranteed because the $\epsilon$-enlarged proximal operator has domain equal to $\mathcal{H}$. Using (\ref{eq:lyap1}), the $x^{k+1}$ - update in I-FBS and the fact that $\lambda_k\leq\lambda_{k+1}$, we infer that
\begin{eqnarray}
\nonumber E_{k+1}-E_k&\leq&\frac{\alpha_{k+1}}{2\lambda_{k+1}}\|\Delta_{k+1}\|^2
-\frac{\alpha_k}{2\lambda_k}\|\Delta_k\|^2
-\frac{1}{\lambda_k}\langle x^{k+1}-y^{k+1},\Delta_{k+1}\rangle
\\\nonumber
&&+\frac{L}{2}\|\Delta_{k+1} - \alpha_k\Delta_k\|^2
+\epsilon_{k}
\\\nonumber
&=&\frac{\alpha_{k+1}}{2\lambda_{k+1}}\|\Delta_{k+1}\|^2
-\frac{\alpha_k}{2\lambda_k}\|\Delta_k\|^2
-\frac{1}{\lambda_k}\langle \Delta_{k+1}-\alpha_k\Delta_k,\Delta_{k+1}\rangle
\nonumber
\\&&\nonumber
+\frac{L}{2}(\|\Delta_{k+1}\|^2+\alpha_k^2\|\Delta_k\|^2)-\alpha_k L\langle \Delta_{k+1},\Delta_k\rangle+\epsilon_{k}\\\nonumber
 &\leq&(\frac{L}{2}-\frac{1}{\lambda_k}+\frac{\alpha_{k+1}}{2\lambda_k})\|\Delta_{k+1}\|^2
 +(\frac{\alpha_k^2 L}{2}-\frac{\alpha_k}{2\lambda_k})\|\Delta_k\|^2 
 \\\nonumber
 &&+ \frac{\alpha_k(1-\lambda_k L)}{\lambda_k}\langle \Delta_{k+1},\Delta_k\rangle+\epsilon_{k}
 \\\nonumber
 &=&-\frac{\alpha_k(1-\lambda_k L)}{2\lambda_k}\|\Delta_{k+1}-\Delta_k\|^2 -\frac{2 +\lambda_k L (\alpha_k-1) - \alpha_k - \alpha_{k+1}}{2\lambda_k }\|\Delta_{k+1}\|^2\\\nonumber
 &&-\frac{L\alpha_k (1-\alpha_k)}{2\lambda_k}\|\Delta_k\|^2+\epsilon_{k}.
 \label{eq:decLyap0}
\end{eqnarray}
Moving terms to the other side and summing implies, for all $N\in\mathbb{Z}_+$,
\begin{eqnarray}
&&\sum_{k=1}^N\left[ \frac{\alpha_k(1-\lambda_k L)}{2\lambda_k}\|\Delta_{k+1}-\Delta_k\|^2 + \frac{2 +\lambda_k L (\alpha_k-1) - \alpha_k - \alpha_{k+1}}{2\lambda_k }\|\Delta_{k+1}\|^2
\right.
\nonumber
\\&&
+\left.\frac{L\alpha_k (1-\alpha_k)}{2\lambda_k}\|\Delta_k\|^2\right]
\nonumber
\\
&\leq& 
E_1 - E_{N+1}+\sum_{k=1}^N\epsilon_{k}
\nonumber
\\
&\leq&
E_1 - F^*+ 
\sum_{k=1}^N\epsilon_{k}
=
F(x^1) - F^*+ 
\frac{\alpha_1}{2\lambda_1}\|\Delta_1\|^2
+\sum_{k=1}^N\epsilon_{k}
<\infty.
\label{eq:sumfinite}
\end{eqnarray}
Inequality (\ref{eq:sumfinite}) along with the assumptions on $\{\alpha_k\}_{k\in\mathbb{N}}$ and $\{\lambda_k\}_{k\in\mathbb{N}}$  imply statement (\ref{state:finitesum}).
 Statement (\ref{state:finitesum}) implies  $\|\Delta_{k+1}\|\to 0$, therefore $\|x^k - y^{k+1}\|\to 0$ via the $y^{k+1}$ - update of I-FBS. This implies that $\|x^k-y^k\|\to 0$, because $\|x^k - y^k\|\leq\|x^{k-1}-y^k\|+\|x^{k-1}-x^k\|$. Finally, using the $x^{k+1}$ -  update of I-FBS we infer that
\begin{eqnarray}
\lim_{k\to\infty}d(\nabla f(y^k)+\partial_{\epsilon_k} g(y^k),0) = \lim_{k\to\infty}d(\nabla f(x^k)+\partial_{\epsilon_k} g(x^k),0)= 0.
\label{eq:gradlimit}
\end{eqnarray}
\subsubsection*{Proof of statement (\ref{state:weak})}
If $x^{v_k}$ is a subsequence which weakly converges to $x'$, then the $y^{k+1}$-update of I-FBS implies $y^{v_k}$ also weakly converges to $x'$. This, combined with the $x^{k+1}$-update implies that $x'\in X^*$. Suppose that for any $x^*\in X^*$, the sequence $\|x^k-x^*\|$ has a limit. This implies the sequence $x^k$ is bounded and therefore it has at least one weakly-convergent subsequence, $x^{\nu_k}$ (ordinary convergence in $\mathbb{R}^n$). By the above reasoning the limit of this subsequence, $\tilde{x}$ must be in $X^*$. Furthermore $\lim_k\|x^k-\tilde{x}\|$ exists. Consider another subsequence $x^{\nu'_k}$ which converges to $\tilde{x}'\in X^*$. By considering the fact that $\lim_k\|x^{\nu_k}-\tilde{x}\|^2=\lim_k\|x^{\nu_k'}-\tilde{x}\|^2$ and the corresponding statement for $\tilde{x}'$, one can see that $\|\tilde{x}-\tilde{x}'\|=0$. Therefore the set of weakly convergent subsequences is the singleton $\{\tilde{x}\}$. Thus $x^k$ weakly converges to $\tilde{x}\in X^*$ (This is Opial's Lemma \cite{opial1967weak}). 

 Assume $X^*$ is non-empty. We now proceed to show that, for any $x^*\in X^*$, the sequence $\|x^k-x^*\|$ has a limit. Our proof closely follows Moudafi and Oliny's analysis \cite{MoudafiOliny}, and is similar to the later variants \cite{chambolle2014weak,lorenz2014inertial}. The main difference is we allow for $\alpha_k$ to be $1$ for a finite number of iterations. Fix $x^*\in X^*$ and define $\varphi_k = \frac{1}{2}\|x^k-x^*\|^2$. Now
\begin{eqnarray}
\varphi_k - \varphi_{k+1} = \frac{1}{2}\|\Delta_{k+1}\|^2+\langle x^{k+1}-y^{k+1},x^*-x^{k+1}\rangle+\alpha_k\langle \Delta_k,x^*-x^{k+1}\rangle.
\label{eq:mod1}
\end{eqnarray}
Since $$-x^{k+1}+y^{k+1}-\lambda_k\nabla f(y^{k+1})\in\lambda_k\partial_{\epsilon_{k}} g(x^{k+1}),$$ $-\lambda_k\nabla f(x^*)\in\lambda_k\partial_{\epsilon_{k}} g(x^*)$ and $\langle\partial_{\epsilon} g(x^{k+1})-\partial_{\epsilon} g(x^*),x^{k+1}-x^*\rangle\geq -\epsilon$, it follows that
\begin{eqnarray}
\langle x^{k+1}-y^{k+1}+\lambda_k(\nabla f (y^{k+1})-\nabla f(x^*)),x^*-x^{k+1}\rangle\geq -\lambda_k\epsilon_{k}.
\label{eq:mod2}
\end{eqnarray}
Combining (\ref{eq:mod1}) and (\ref{eq:mod2}) we obtain
\begin{eqnarray}
\varphi_k - \varphi_{k+1}&\geq&\frac{1}{2}\|\Delta_{k+1}\|^2+\lambda_k\langle\nabla f(y^{k+1})-\nabla f(x^*),x^{k+1}-x^*\rangle - \alpha_k\langle \Delta_k,x^{k+1}-x^*\rangle
\nonumber\\
&&-\lambda_k\epsilon_{k}.
\label{eq:mod3}
\end{eqnarray}
Now
\begin{eqnarray}
\nonumber\langle \Delta_k,x^{k+1}-x^*\rangle&=&\langle \Delta_k,x^k-x^*\rangle+\langle \Delta_k,\Delta_{k+1}\rangle\\
&=&\varphi_k - \varphi_{k-1} + \frac{1}{2}\|\Delta_k\|^2+\langle \Delta_k,\Delta_{k+1}\rangle.
\label{eq:mod4}
\end{eqnarray}
Combining (\ref{eq:mod3}) and (\ref{eq:mod4}) yields
\begin{eqnarray}
\nonumber\varphi_{k+1}-\varphi_k - \alpha_k(\varphi_k - \varphi_{k-1})&\leq& -\frac{1}{2}\|\Delta_{k+1}\|^2+\alpha_k\langle \Delta_k,\Delta_{k+1}\rangle+\frac{\alpha_k}{2}\|\Delta_k\|^2\\&&-\lambda_k\langle\nabla f(y^{k+1})-\nabla f(x^*),x^{k+1}-x^*\rangle 
\nonumber\\
&&+ \lambda_k\epsilon_{k}.
\label{eq:mod5}
\end{eqnarray}
Now we use the fact that $\nabla f$ is cocoercive as follows. Inequality (\ref{eq:coercdef}) implies 
\begin{eqnarray}
\nonumber\lambda_k\langle\nabla f(y^{k+1})-\nabla f(x^*),x^{k+1}-x^*\rangle &=&\lambda_k(\langle\nabla f(y^{k+1})-\nabla f(x^*),y^{k+1}-x^*\rangle
\\\nonumber
&&+\langle\nabla f(y^{k+1}-\nabla f(x^*),x^{k+1}-y^{k+1}\rangle)\\\nonumber
&\geq& 
\frac{\lambda_k}{L}(\|\nabla f(y^{k+1})-\nabla f(x^*))\|^2 
\\\nonumber
&&+ \langle\nabla f(y^{k+1})-\nabla f(x^*),x^{k+1}-y^{k+1}\rangle)\\
&\geq&-\frac{\lambda_k L}{4}\|x^{k+1}-y^{k+1}\|^2
\label{eq:coerc}
\end{eqnarray}
where (\ref{eq:coerc}) follows by completing the square. Combining (\ref{eq:mod5}) and (\ref{eq:coerc}) we infer
\begin{eqnarray}
\nonumber\varphi_{k+1}-\varphi_k-\alpha_k(\varphi_k - \varphi_{k-1})
&\leq& 
-\frac{1}{2}\|\Delta_{k+1}\|^2+\alpha_k\langle \Delta_k,\Delta_{k+1}\rangle+\frac{\alpha_k}{2}\|\Delta_k\|^2\\\nonumber&&+\frac{\lambda_k L}{4}\|\Delta_{k+1}-\alpha_k\Delta_k\|^2+ \lambda_k\epsilon_{k}\\\nonumber
&=& 
\frac{\alpha_k}{2}\left(\frac{\lambda L}{2}-1\right)\|\Delta_{k+1}-\Delta_k\|^2
\\
&&
+
\frac{\alpha_k}{4}\left(4+\lambda L (\alpha_k- 1)\right)\|\Delta_k\|^2
\nonumber\\
&&
+
\frac{\lambda L - 2}{4}(1-\alpha_k)\|\Delta_{k+1}\|^2
+ \lambda_k\epsilon_{k}.
\label{eq:recurse1}
\end{eqnarray}
Note that the coefficients of $\|\Delta_{k+1}-\Delta_k\|^2$ and $\|\Delta_{k+1}\|^2$ are non-positive. Set $\theta_k\triangleq \varphi_k - \varphi_{k-1}$ and 
\begin{eqnarray}
\delta_k\triangleq 
\frac{\alpha_k}{4}\left(4+\lambda L (\alpha_k- 1)\right)\|\Delta_k\|^2
+\lambda_k\epsilon_{k}
\label{eq:defofdelt}
\end{eqnarray}
 and note that $\sum_{k=1}^\infty \delta_k < \infty$. 

The argument from now on is basically identical to \cite{MoudafiOliny} except we allow for sequences $\alpha_k$ which are equal to $1$ for a finite number of $k$. Restate (\ref{eq:recurse1}) as
\begin{eqnarray}
\nonumber
\theta_{k+1}&\leq& \alpha_k\theta_k + \delta_k\\
&\leq& \alpha_k[\theta_k]_+ + \delta_k.
\label{eq:recurs2}
\end{eqnarray}
Since $\lim\sup\alpha_k<1$, there exists an integer $K\geq0$ and $\overline{\alpha}\in[0,1)$ such that $\alpha_k\leq\overline{\alpha}<1$ for all $k>K$. This and (\ref{eq:recurs2}) imply that, for $k>K$
\begin{eqnarray*}
[\theta_{k+1}]_+\leq\overline{\alpha}[\theta_k]_++\delta_k.
\end{eqnarray*}
Thus for $k>K$
\begin{eqnarray*}
[\theta_{k+1}]_+
\leq
\overline{\alpha}^{k-K}[\theta_K]_+
+\sum_{j=K}^{k}\overline{\alpha}^{k-j}\delta_{j}
+\overline{\alpha}^{k-K}\sum_{j=1}^K\delta_j.
\end{eqnarray*}
Careful examination of this expression yields
\begin{eqnarray}
\sum_{k=0}^{\infty}[\theta_{k+1}]_+
\leq
K\sum_{k=0}^K\delta_k +
\frac{\overline{\alpha}^K}{1-\overline{\alpha}}\left([\theta_1]_++\sum_{k=K}^{\infty}\delta_k\right)<\infty.
\label{eq:upperonsumtheta}
\end{eqnarray}
Set $w_k\triangleq \varphi_k - \sum_{j=0}^k[\theta_j]_+$. Since $\varphi_k\geq 0$ and $\sum[\phi_j]_+<\infty$, $w_k$ is bounded from below. $w_k$ is non-increasing, therefore we have it converges. Therefore $\varphi_k$ converges for every $x^*\in X^*$. By invoking Opial's Lemma, statement (vi) is established.
\end{proof}


Theorem \ref{thm:fista_lyap} does not apply to the FISTA-like parameter choices because for all of these choices $\alpha_k\to 1$.  However the theorem does apply if we make the following modification. Replace the momentum parameter sequence $\{\alpha_k\}_{k\in\mathbb{N}}$ with $\min(\alpha_k,\overline{\alpha})$ where $\overline{\alpha}<1$. This parameter choice satisfies the assumptions of the theorem, and $\overline{\alpha}$ can be chosen arbitrarily close to $1$. However the $O(1/k^2)$ objective function convergence rate is no longer guaranteed once $\alpha_k$ exceeds $\overline{\alpha}$.


In the following Corollary, we use (\ref{eq:sumfinite}) to determine explicit bounds on $\sum_k\|\Delta_k\|^2$ which will be useful in the analysis of Problem SO.
\begin{corollary}
\label{cor1st}
Suppose that Assumption 1 holds. Assume $\{\lambda_k\}_{k\in\mathbb{N}}$ is non-decreasing and satisfies $0<\lambda_k\leq 1/L$ for all $k$, $\{\epsilon_k\}_{k\in\mathbb{N}}$ satisfies $\epsilon_k\geq0$ for all $k$ and $\sum_k\epsilon_k<\infty$, there exists $\overline{\alpha}\in[0,1)$ such that $\{\alpha_k\}_{k\in\mathbb{N}}$ satisfies $0\leq\alpha_k\leq\overline{\alpha}$ for all $k$. Then for the iterates of I-FBS,
\begin{eqnarray}
\sum_{k=1}^\infty\|\Delta_k\|^2 
\leq 
\frac{2}{2L(1-\overline{\alpha})-1}
\lbr
F(x^1)-F^*
+\frac{\alpha_1}{2\lambda_1}\|\Delta_1\|^2
+\sum_{k=1}^\infty\epsilon_k
\rbr
.
\label{co1}
\end{eqnarray}
If, in addition, there exists $\underline{\alpha}\in[0,\overline{\alpha}]$ such that $\alpha_k\geq\underline{\alpha}$ for all $k$, then
\begin{eqnarray}
\sum_{k=1}^\infty\|\Delta_k\|^2 
&\leq& 
\frac{2}{L^2\underline{\alpha}(1-\overline{\alpha})}
\lbr
F(x^1)-F^*
+\frac{\alpha_1}{2\lambda_1}\|\Delta_1\|^2
+\sum_{k=1}^\infty\epsilon_k
\rbr.
\label{co2}
\end{eqnarray}
\end{corollary}
\begin{proof}
Inequality (\ref{eq:sumfinite}) implies
\begin{eqnarray}
\nonumber
F(x^1)-F^*+\frac{\alpha_1}{2\lambda_1}\|\Delta_1\|^2+\sum_{k=1}^\infty\epsilon_k
&\geq&
\sum_{k=1}^\infty
\frac{2+\lambda_k(\alpha_k-1)-\alpha_k-\alpha_{k+1}}{2\lambda_k}
\|\Delta_{k+1}\|^2
\\\label{ppo}
&\geq&
\sum_{k=1}^\infty\frac{2-2\overline{\alpha}-\lambda_k}{2\lambda_k}\|\Delta_{k+1}\|^2
\\\label{ppp}
&\geq&
\sum_{k=1}^\infty\frac{2L(1-\overline{\alpha})-1}{2}\|\Delta_{k+1}\|^2
\end{eqnarray}
 which proves (\ref{co1}). To derive (\ref{ppo}) we used the fact that $0\leq\alpha_k\leq\overline{\alpha}$. To derive (\ref{ppp}) we used the fact that $\lambda L\leq 1$. 
 
 Inequality (\ref{eq:sumfinite}) also implies
 \begin{eqnarray}
 \nonumber
 F(x^1)-F^*+\frac{\alpha_1}{2\lambda_1}\|\Delta_1\|^2+\sum_{k=1}^\infty\epsilon_k
 &\geq&
 \sum_{k=1}^\infty\frac{L\alpha_k(1-\alpha_k)}{2\lambda_k}\|\Delta_k\|^2
 \\\nonumber
 &\geq&
 \sum_{k=1}^\infty
 \frac{L^2\underline{\alpha}(1-\overline{\alpha})}{2}\|\Delta_k\|^2
 \end{eqnarray}
 which proves (\ref{co2}). 
\end{proof}

\section{Convergence Analysis of I-FBS for Sparse Optimization}
\label{sec:l1ls}\label{sec:convFistal1ls}
\subsection{Finite Convergence Results}
\label{sec:fistal1ls_finite}
We now turn our attention to Problem SO. The following theorem proves finite convergence to $0$ for the components in $D$, and finite convergence to the correct sign for the components in $E$ (recall the definitions of $D$ and $E$ in Section \ref{sec:propsparse}). Following the terminology of \cite{liang2014local} we will refer to this as the ``finite manifold identification period". The manifold in this case is the half-space of vectors with support a subset of $E$ and non-zero components with sign $-h_i^*/\rho$.
This theorem generalizes the result of Theorem 4.5 in \cite{Hale:2008} from ISTA to I-FBS. For simplicity, we only consider the case where $\epsilon_k$ is $0$ for all $k$, meaning the proximal operator is computed exactly. Thus the results are stated for I-FBS not I-FBS-$\epsilon$. Note that the proximal operator w.r.t. the $\ell_1$ norm is relatively easy to compute as it is in seperable closed form, thus we do not think it is worth considering I-FBS-$\epsilon$ in this case. In the next subsection, we consider the FISTA-like methods. 

\begin{theorem}
Suppose that Assumption SO holds. Assume $\{\lambda_k\}_{k\in\mathbb{N}}$ is non-decreasing and satisfies $0<\lambda_k\leq1/L$, and there exists $\underline{\alpha},\overline{\alpha}\in[0,1)$ such that $\{\alpha_k\}_{k\in\mathbb{N}}$ satisfies $\underline{\alpha}\leq\alpha_k\leq \overline{\alpha}$ for all $k$. Then, there exist constants $K_D>0$ and $K_E>0$ such that the iterates of I-FBS applied to Problem SO satisfy, for all $k>K_E$,
\begin{eqnarray}
\label{thm:finite:eq2}
\sgn\left(y_i^{k}-\lambda_k\nabla f(y^{k})_i\right)
&=&-\frac{h^*_i}{\rho},\ \forall i \in E,
\end{eqnarray}
and, for all $k>K_D$
\begin{eqnarray}
x^k_i =y^k_i= 0,\ \forall i \in D.
\label{thm:finite:eq1}
\end{eqnarray}
Furthermore, $K_E\leq\overline{K}_E$ and $K_D \leq\overline{K}_D$, where 
\begin{eqnarray}
\overline{K}_E
&\triangleq&
\frac{1}{\rho^2\lambda_1^2}
\left[
\frac{2\overline{\alpha}(1+\overline{\alpha})\left(F(x^1)-F^*+\frac{\alpha_1}{2\lambda_1}\|\Delta_1\|^2\right)}{\underline{\alpha}(1-\overline{\alpha})L^2}
+\|x^1-x^*\|^2-\overline{\alpha}\|x^0-x^*\|^2
\right]
\nonumber\\
&&
+\frac{\underline{\alpha}}{1-\underline{\alpha}}
\label{eq:maxsign}
\end{eqnarray}
and 
\begin{eqnarray}
\overline{K}_D
&\triangleq&
\frac{1}{\omega^2\lambda_1^2}
\left[
\frac{2\overline{\alpha}(1+\overline{\alpha})\left(F(x^1)-F^*+\frac{\alpha_1}{2\lambda_1}\|\Delta_1\|^2\right)}{\underline{\alpha}(1-\overline{\alpha})L^2}
+\|x^1-x^*\|^2-\overline{\alpha}\|x^0-x^*\|^2
\right]
\nonumber\\&&
+\frac{\underline{\alpha}}{1-\underline{\alpha}}+2
\label{eq:maxsupport}
\end{eqnarray}
for any $x^*\in X^*$.

\label{thm:finite}
\end{theorem}
\begin{proof}
Note that this parameter choice satisfies the requirements of Theorem \ref{thm:fista_lyap} and Corollary \ref{cor1st}. Furthermore, by assumption, $X^*$ is non-empty and $F^*\geq-\infty$, thus all conclusions of Theorem \ref{thm:fista_lyap} and Corollary \ref{cor1st} hold. Throughout the proof, fix an arbitrary $x^*\in X^*$.

\subsubsection*{Proof of (\ref{thm:finite:eq2})} 
Fix a $\lambda>0$. Recall from Theorem \ref{thm:constGrad} there exists a vector $h^*$ such that $\nabla f(x^*) = h^*$ for all $x^*\in X^*$, and that $\supp (x^*)\subset E$. For $i\in\supp (x^*)$,
\begin{eqnarray}
0\neq x_i^* = \sgn\left(x_i^*-\lambda h^*_i\right)\left[|x_i^*-\lambda h^*_i|-\rho\lambda\right]_+.
\label{eq:fix_point}
\end{eqnarray}
Therefore $|x_i^*-\lambda h^*_i|>\rho\lambda$ for all $i\in\supp(x^*)$. On the other hand, if $i\in E\setminus \supp(x^*)$, then 
\begin{eqnarray*}
|x_i^*-\lambda h^*_i|=\lambda|h_i^*|=\rho\lambda.
\end{eqnarray*}
Therefore 
\begin{eqnarray*}
|x_i^*-\lambda h^*_i|\geq\rho\lambda,\ \forall i \in E.
\end{eqnarray*}
Looking at (\ref{eq:fix_point}) it can be seen that
\begin{eqnarray}
\sgn(x^*_i)=\sgn(x_i^*-\lambda h^*_i),\quad\forall i\in \supp(x^*).
\label{eq:l1ls2}
\end{eqnarray}
 Note by Theorem \ref{thm:constGrad}, if $i\in \supp(x^*)$, then $\sgn(x_i^*)=-h_i^*/\rho$. Else if $i\in E\setminus \supp(x^*)$ then
 \begin{eqnarray}
 \sgn(x_i^*-\lambda h^*_i)= \sgn(-\lambda h_i^*)
                   =-\sgn(h_i^*)
                   = -\frac{h_i^*}{\rho}.\label{eq:l1ls1}
 \end{eqnarray}   
 Combining (\ref{eq:l1ls2}) and (\ref{eq:l1ls1}) yields 
 \begin{eqnarray}
 \sgn(x_i^*-\lambda h^*_i)= -h_i^*/\rho\quad \forall \ i\in E, \lambda>0.
 \label{eq:callback}
 \end{eqnarray}
 
 Let $\nu_k=\rho\lambda_k$. If 
\begin{eqnarray}
\sgn\left(y^{k+1}_i-\lambda_k\nabla f(y^{k+1})_i\right)\neq \sgn(x_i^*-\lambda_k h^*_i)=-h_i^*/\rho\quad\hbox{for some $i\in E$},
\label{eq:sgn}
\end{eqnarray}
then Lemma \ref{lemma:Sv} implies 
\begin{eqnarray}
|x_i^{k+1}-x_i^*|^2&=&\left|S_{\nu_k} (y_i^{k+1}-\lambda_k\nabla f(y^{k+1})_i)-S_{\nu_k} (x_i^*-\lambda_k h^*_i)\right|^2\nonumber
\\\nonumber
&\leq&\left(|y_i^{k+1}-\lambda_k\nabla f(y^{k+1})_i-(x_i^*-\lambda_k h^*_i)|-\nu_k\right)^2\\
&\leq&\left|y_i^{k+1}-\lambda_k\nabla f(y^{k+1})_i-(x_i^*-\lambda_k h^*_i)\right|^2-\nu_k^2 \label{eq:nice}
\label{eq:mid}
\end{eqnarray}
where (\ref{eq:nice}) follows because $$|(y_i^{k+1}-\lambda_k\nabla f(y^{k+1})_i)-(x_i^*-\lambda_k h^*_i)|>|(x_i^*-\lambda_k h^*_i)|\geq\nu_k>0.$$ 
 
 Using (\ref{eq:mid}) we can say the following: Condition (\ref{eq:sgn}) implies that
\begin{eqnarray}
\nonumber\|x^{k+1}-x^*\|^2&=&\sum_{j\neq i}|x^{k+1}_j-x^*|^2 +|x_i^{k+1}-x^*|^2\\\label{eq:l1lsnonx}
&\leq &\sum_{j\neq i}\left|y_j^{k+1}-\lambda_k\nabla f(y^{k+1})_j-(x_j^*-\lambda_k h^*_j)\right|
\nonumber\\&&
+\left|y_i^{k+1}-\lambda_k\nabla f(y^{k+1})_i-(x_i^*-\lambda_k h^*_i)\right|^2-\nu_k^2\\\nonumber
&\leq&\|y^{k+1}-\lambda_k\nabla f(y^{k+1}) - (x^*-\lambda_k h^*)\|^2-\nu_k^2\\                 &\leq&\|y^{k+1}-x^*\|^2-\nu_k^2\label{eq:nonexpand}\\\nonumber
&=&\|x^k+\alpha_k\Delta_{k}-x^*\|^2-\nu_1^2
\\
&=&\|x^k-x^*\|^2+\alpha_k^2\|\Delta_{k}\|^2+2\alpha_k\langle \Delta_{k},x^k-x^*\rangle-\nu_1^2.
\label{eq:thebiggy}
\end{eqnarray}
Inequality (\ref{eq:l1lsnonx}) follows from the element-wise nonexpansiveness of $S_\nu$ along with (\ref{eq:mid}). To deduce (\ref{eq:nonexpand}), we used Lemma \ref{lemma:nonexpansive}. Finally, (\ref{eq:thebiggy}) follows because $\{\lambda_k\}_{k\in\mathbb{N}}$ is non-decreasing and therefore so is $\{\nu_k\}$.

Recall the definition of $\varphi_k \triangleq \frac{1}{2}\|x^k-x^*\|^2$ and $\theta_k\triangleq\varphi_k - \varphi_{k-1}$. Now, moving $\langle\Delta_k,\Delta_{k+1}\rangle$ to the other side of (\ref{eq:mod4}) reveals
\begin{eqnarray}
\langle\Delta_k,x^k - x^*\rangle = \varphi_k - \varphi_{k-1} + \frac{1}{2}\|\Delta_k\|^2.
\label{thebigbiggy}
\end{eqnarray}
 Substituting (\ref{thebigbiggy}) into (\ref{eq:thebiggy}) yields

\begin{eqnarray*}
2(\varphi_{k+1}-\varphi_k)
&\leq&
2\alpha_k(\varphi_{k}-\varphi_{k-1})
+\overline{\alpha}(1+\overline{\alpha})\|\Delta_k\|^2-\nu_1^2,
\label{blah}
\end{eqnarray*}
therefore
\begin{eqnarray*}
\theta_{k+1}
\leq
\alpha_k\theta_k
 + \frac{\overline{\alpha}(1+\overline{\alpha})}{2}\|\Delta_k\|^2 - \frac{\nu_1^2}{2}.
\end{eqnarray*}
Repeating the arguments that led to (\ref{eq:upperonsumtheta}), we can say the following: if (\ref{eq:sgn}) is true then
\begin{eqnarray}
\theta_{k+1}
\leq
\overline{\alpha}^k\theta_1+
\frac{\overline{\alpha}(1+\overline{\alpha})}{2}\sum_{j=1}^k\overline{\alpha}^{k-j}\|\Delta_j\|^2 
-\frac{\nu^2}{2}\sum_{j=1}^k\underline{\alpha}^{k-j}.
\end{eqnarray}
Therefore, for $M\in\mathbb{Z}_+$, if (\ref{eq:sgn}) holds at iteration $M$, then
\begin{eqnarray}
\nonumber
\varphi_M -\varphi_0
&=& 
\sum_{k=1}^M\theta_{k}
\\\nonumber
&\leq&
\frac{\theta_1(1-\overline{\alpha}^M)}{1-\overline{\alpha}}+
\frac{\overline{\alpha}(1+\overline{\alpha})}{2(1-\overline{\alpha})}\sum_{k=1}^M\|\Delta_k\|^2
-\frac{\nu^2}{2}\sum_{k=1}^M\sum_{j=0}^{k-1}\underline{\alpha}^j
\\\label{a}
&\leq&
\frac{\theta_1}{1-\overline{\alpha}}+
\frac{\overline{\alpha}(1+\overline{\alpha})}{2(1-\overline{\alpha})}\sum_{k=1}^M\|\Delta_k\|^2
-\frac{\nu^2}{2}
\lbr 
\frac{M}{1-\underline{\alpha}} - \frac{\underline{\alpha}}{(1-\underline{\alpha})^2}
\rbr.
\end{eqnarray} 
To derive (\ref{a}) we lower bounded the coefficient of $-\frac{\nu^2}{2}$. Since $\varphi_k\geq 0$, if (\ref{eq:sgn}) is true at iteration $k$ then
\begin{eqnarray}
\label{apa}
k&\leq&
\frac{2(1-\underline{\alpha})}{\nu_1^2}
\left[
\frac{\overline{\alpha}(1+\overline{\alpha})}{2(1-\overline{\alpha})}
\sum_{k=1}^\infty\|\Delta_k\|^2
+\frac{\|x^0-x^*\|^2}{2}+\frac{\theta_1}{1-\overline{\alpha}}
\right]
+\frac{\underline{\alpha}}{1-\underline{\alpha}}
\\
&\leq&
\frac{1}{\nu_1^2}
\left[
\frac{2\overline{\alpha}(1+\overline{\alpha})\left(F(x^1)-F^*+\frac{\alpha_1}{2\lambda_1}\|\Delta_1\|^2\right)}{\underline{\alpha}(1-\overline{\alpha})L^2}
+\|x^1-x^*\|^2-\overline{\alpha}\|x^0-x^*\|^2
\right]
\nonumber\\
&&+\frac{\underline{\alpha}}{1-\underline{\alpha}}
\label{aad}
\end{eqnarray} 
To derive ({\ref{aad}}) we used the upper bound on $\sum_k\|\Delta_k\|^2$ in (\ref{co2}) from Corollary \ref{cor1st}. This upper bound is tighter than the other upper bound for $\sum_k \|\Delta_k\|^2$ given in (\ref{co1}), so long as $L>2/\underline{\alpha}$.

\subsubsection*{Proof of (\ref{thm:finite:eq1})} 
 Recall the definition of $\omega$ and note that 
 \begin{eqnarray}
 \lambda_k\omega &=&\min\{\nu_k-|x_i^* - \lambda_k h^*_i|:i\in D\}>0
 \label{mm}
 \end{eqnarray}
Consider $i\in D$ (which implies $i\notin\supp(x^*)$). If $x_i^{k+1}\neq0$, then Lemma \ref{lemma:Sv} implies
\begin{eqnarray}
\nonumber
|x_i^{k+1}|^2
&=& 
|S_{\nu_{k}}\lbr y_i^{k+1}-\lambda_k\nabla f(y^{k+1})_i\rbr 
- S_{\nu_{k}}\lbr x^*_i - \lambda_k h_i^*\rbr|^2
\\\nonumber
&\leq&
\left[
|(y_i^{k+1}-\lambda_k\nabla f(y^{k+1}_i))-(x_i^*-\lambda_k h^*_i)|
-\lbr\nu_k - | x_i^* - \lambda_k h_i^*|\rbr
\right]^2
\\\label{mnn}
&\leq&
|(y_i^{k+1}-\lambda_k\nabla f(y^{k+1}_i))-(x_i^*-\lambda_k h^*_i)|^2
-\lbr\nu_k - | x_i^* - \lambda_k h_i^*|\rbr^2
\\\label{mmn}
&\leq&
|(y_i^{k+1}-\lambda_k\nabla f(y^{k+1}_i))-(x_i^*-\lambda_k h^*_i)|^2
- \omega^2\lambda^2_k.
\end{eqnarray}
To derive (\ref{mnn}) we used the fact that $$|(y_i^{k+1}-\lambda_k\nabla f(y^{k+1}_i))+\lambda_k h^*_i|>\nu_k - \lambda_k| h_i^*|.$$ To derive (\ref{mmn}) we used (\ref{mm}).  Repeating the arguments used to prove (\ref{eq:thebiggy}) we can say the following. If there exists $i\in D$, such that $x_i^{k+1}\neq 0$, then $k\leq \overline{K}_{D}-2$, with $\overline{K}_D$ defined in (\ref{eq:maxsupport}). Therefore $|x_i^k|=0$ for all $k>\overline{K}_D-2$. Since $y_i^{k+1}=x_i^k+\alpha_k(x_i^k-x_i^{k-1})$, $y_i^k=0$ for all $i\in D$ and $k>\overline{K}_D-2+2=\overline{K}_D$, which proves (\ref{thm:finite:eq1}).
\end{proof}

Note that if  $x_i^{k+1}\neq 0$ then $\sgn(x_i^{k+1})=y_i^{k}-\lambda_k\nabla f(y^{k})_i$, thus (\ref{thm:finite:eq2}) implies convergence in sign. We can recover the result by Hale et al. for ISTA (Theorem 4.5 \cite{Hale:2008}). To see this, consider (\ref{apa}) with $\overline{\alpha}=\underline{\alpha}=0$ and then use the upper bound on $\sum_k\|\Delta_k\|^2$ given in (\ref{co1}) of Corollary \ref{cor1st}. Note that we defined the constant $\omega$ in a slightly differently way to Hale et al.
\subsection{Finite Convergence of FISTA}
We can prove convergence to the optimal manifold in a finite number of iterations under more general conditions than required in Theorem \ref{thm:finite}, however without explicit bounds on the number of iterations. A corollary of the following theorem is that the FISTA-like choices proposed by Beck and Teboulle \cite{Beck:2009} and Chambolle and Dossal \cite{chambolle2014weak} achieve finite manifold identification. 
\begin{theorem}
Suppose that Assumption SO holds. Assume $\{\lambda_k\}_{k\in\mathbb{N}}$ is non-decreasing and satisfies $0<\lambda_k< 2/L$ for all $k$, and there exists $\overline{\alpha}\geq 0 $ such that $\{\alpha_k\}_{k\in\mathbb{N}}$ satisfies $0\leq\alpha_k\leq \overline{\alpha}$, for all $k$. If, for the iterates of I-FBS applied to Problem SO, it is true that $\sum_{k=1}^\infty\|\Delta_k\|^2<\infty$ and $\|x^k-x^*\|$ is bounded for some $x^*\in X^*$ and for all $k$, then there exists a constant $K>0$ such that  for the iterates of I-FBS (\ref{thm:finite:eq2}) and (\ref{thm:finite:eq1}) hold for all $k>K$. 
\label{cor:forchamb}
\end{theorem}
\begin{proof}
Inequality (\ref{eq:thebiggy}) and the equivalent recursion for when (\ref{eq:sgn}) holds can be proved in exactly the same way. However, we cannot rely on $\overline{\alpha}<1$, so we have to modify the proof from that point onwards. Once again, fix $x^*\in X^*$. Rewriting (\ref{eq:thebiggy}), we can say that (\ref{eq:sgn})  implies that
\begin{eqnarray}
\nonumber
\|x^{k+1}-x^*\|^2
&\leq&\|x^k-x^*\|^2+\overline{\alpha}^2\|\Delta_{k}\|^2+2\overline{\alpha}\langle \Delta_{k},x^k-x^*\rangle-\nu_1^2
\\\label{afs}
&\leq&\|x^k-x^*\|^2+\overline{\alpha}^2\|\Delta_{k}\|^2+2\overline{\alpha}\|\Delta_{k}\|\|x^k-x^*\|-\nu_1^2
\\\label{afss}
&\leq&\|x^k-x^*\|^2+\overline{\alpha}^2\|\Delta_{k}\|^2+2\overline{\alpha}M_1\|\Delta_{k}\|-\nu_1^2
\end{eqnarray}

 where we used the Cauchy-Schwarz inequality to get (\ref{afs}). To derive (\ref{afss}) we used the assumption that there exists $M_1>0$ such that $\|x^k-x^*\|<M_1$. Also by assumption, there exists $M_2>0$ such that
\begin{eqnarray}
\sum_{k=1}^\infty\|\Delta_k\|^2<M_2.
\label{asas}
\end{eqnarray}
Now, by Jensen's inequality
\begin{eqnarray}
\nonumber
\sum_{i=0}^k\|\Delta_i\|
&\leq&  
\sqrt{k\sum_{i=0}^k\|\Delta_i\|^2}
\\
&\leq&
\sqrt{M_2 k}.
\label{asfas}
\end{eqnarray}
Substituting (\ref{asas}) and (\ref{asfas}) into (\ref{afss}) yields the following: If (\ref{eq:sgn}) is true then
\begin{eqnarray}
\|x^{k+1}-x^*\|^2
&\leq&
\|x^0-x^*\|^2+\overline{\alpha}^2M_2+2\overline{\alpha}M_1\sqrt{M_2 k}-k\nu_1^2
\label{bigggg}
\end{eqnarray}
The r.h.s. of (\ref{bigggg}) can be non-negative for only a finite number of iterations, which proves (\ref{thm:finite:eq2}). 

For (\ref{thm:finite:eq1}), the recursion is 
\begin{eqnarray*}
\|x^{k+1}-x^*\|^2
&\leq&\|x^k-x^*\|^2+\overline{\alpha}^2\|\Delta_{k}\|^2+2\overline{\alpha}\langle \Delta_{k},x^k-x^*\rangle-\omega^2\lambda_1^2
\end{eqnarray*}
and the reasoning is the same from this point onwards. 
\end{proof}

The classical FISTA parameter choice in (\ref{eq:fista_mo2}) due to Beck and Teboulle \cite{Beck:2009}, along with others which guarantee the $O(1/k^2)$ rate provided by Tseng \cite{tseng2008accelerated}, satisfy the assumptions of Theorem \ref{cor:forchamb}, so long as $F$ is coercive (or equivalently has bounded level-sets). The first condition, $\sum_k\|\Delta_k\|^2<\infty$ can be shown by considering the following facts. The sequence $b_k$ defined on page 196 of \cite{Beck:2009} is bounded, which implies the sequence $u_k$ defined on page 194 of \cite{Beck:2009} is also bounded. If $F$ is coercive, then $x^k$ is bounded, since $F(x^k)\to F^*$. This implies $\|x^k-x^*\|$ is bounded for some $x^*\in X^*$. It also implies $t_k\Delta_k$ is bounded and since $t_k=O(k)$, $\|\Delta_k\|=O(1/k)$, and the result follows\footnote{We thank Antonin Chambolle and Charles Dossal for pointing this out to us.}.

 The parameter choice (\ref{eq:fista_mo3}) due to Chambolle and Dossal \cite{chambolle2014weak} satisfies the assumptions of this theorem, even when $F$ is not assumed to be coercive. $\sum_k\|\Delta_k\|^2$ is finite by Corollary 2 of \cite{chambolle2014weak} and $\|x^k-x^*\|$ is shown to be bounded for all $x^*\in X^*$ in the proof of Theorem 3 of \cite{chambolle2014weak}. (In fact Chambolle and Dossal proved that $\sum_k k\|\Delta_k\|^2$ is finite for their parameter choice.)

\subsection{Finite Reduction to Local Minimization}
Theorems \ref{thm:finite} and \ref{cor:forchamb} allow us to characterize the behavior of I-FBS (including the FISTA-like choices) after a finite manifold identification period. In the following corollary, we show that after a finite number of iterations, I-FBS reduces to minimizing a smooth function over $E$ subject to an orthant constraint. The following corollary generalizes the result of Corollary 4.6 in \cite{Hale:2008} from ISTA to I-FBS.

\begin{corollary}
\label{cor:q_lin}
 Suppose that Assumption SO holds. Assume $\{\lambda_k\}_{k\in\mathbb{N}}$ is non-decreasing and satisfies $0<\lambda_k\leq1/L$, and there exists $\underline{\alpha},\overline{\alpha}\in[0,1)$ such that $\underline{\alpha}\leq\overline{\alpha}$ and $\{\alpha_k\}_{k\in\mathbb{N}}$ satisfies $\underline{\alpha}\leq\alpha_k\leq \overline{\alpha}$ for all $k$. Then, after finitely many iterations, the iterates of I-FBS applied to Problem SO become equivalent to the iterates of I-FBS applied to minimizing $\phi:\mathbb{R}^{|E|}\to\mathbb{R}$, where
\begin{eqnarray}
\phi(x_E)&\triangleq&-(h_E^*)^T x_E+f\left((x_E,0)\right),
\label{eq:phidef}
\end{eqnarray}
constrained to the orthant $O_E$, where
\begin{eqnarray}
O_E&\triangleq&\{x_E\in\mathbb{R}^{|E|}:-\sgn(h_i^*)\, x_i\geq 0,\ \forall i\in E\}.
\label{O_E_def}
\end{eqnarray}
Specifically, there exists $K>0$ such that for all $k>K$, 
\begin{eqnarray}
\label{eq:proj_HB1}
y_E^{k+1} &=& x_E^k+\alpha_k(x_E^k-x_E^{k-1})\\
x_E^{k+1}&=& P_{O_E}\lbr y_E^{k+1}-\lambda_k\nabla\phi(y_E^{k+1})\rbr,
\label{eq:proj_HB2}
\end{eqnarray}
$x_D^k=y_D^k=0$, and $F(x^k) = \phi(x_E^k)$. Furthermore $K\leq\max\{\overline{K}_D, \overline{K}_E\}$, with $\overline{K}_D$ and $\overline{K}_E$ defined in (\ref{eq:maxsign}) and (\ref{eq:maxsupport}). 

Alternatively, if the conditions of Theorem \ref{cor:forchamb} hold, then there exists $K'>0$ such that (\ref{eq:proj_HB1})-(\ref{eq:proj_HB2}) hold, $x_D^k=y_D^k=0$, and $F(x^k) = \phi(x_E^k)$, for all $k>K'$.  
\end{corollary}

\begin{proof}
From Theorem \ref{thm:finite}, there exists a $K$ such that for all $k>K$, (\ref{thm:finite:eq2}) and (\ref{thm:finite:eq1}) hold and  $K\leq\max\{\overline{K}_D, \overline{K}_E\}$. Take $k>K$. Since $x_i^k=0$ for all $i\in D$ it suffices to consider $i\in E$. For $i\in E$, $k>K$, using (\ref{thm:finite:eq2}) we have 
\begin{eqnarray}
x_i^k\geq 0\hbox{ if }\sgn\left(y_i^{k+1}-\lambda_k\nabla f(y^{k+1})_i\right)=1\hbox{ equivalently }h_i^*<0\nonumber
\end{eqnarray}
 and 
\begin{eqnarray}
x_i^k\leq 0\hbox{ if }\sgn\left(y_i^{k+1}-\lambda_k\nabla f(y^{k+1})_i\right)=-1\hbox{ equivalently }h_i^*>0.\nonumber
\end{eqnarray}
Therefore for any $i\in E$, $-h_i^*x_i^k\geq 0$, thus $x^k_E\in O_E$ for all $k>K$. Next note that, for all $i\in E$,
$
-h_i^* x^k_i
=
\rho|x_i|.
$
Therefore for $k>K$,$-(h^*)^Tx_E^k = \rho\|x^k\|_1$, thus $F(x^k)=\phi(x^k_E)$.

Now for $i\in E$, $k>K$, we calculate the quantity
\begin{eqnarray*}
z_i^{k+1}&\triangleq& y_i^{k+1}-\lambda_k\nabla\phi(y_E^{k+1})_i\\
&=& y_i^{k+1}-\lambda_k(-h^*_i+\nabla f(y^{k+1})_i)\\
&=& y_i^{k+1}-\lambda_k \nabla f(y^{k+1})_i+\rho\lambda_k(\frac{h_i^*}{\rho})\\
&=&\sgn\left(y_i^{k+1}-\lambda_k\nabla f(y^{k+1})_i\right)(|y_i^{k+1}-\lambda_k \nabla f(y^{k+1})_i|-\rho\lambda_k).
\end{eqnarray*}
Therefore, for $i\in E$, $k>K$,
\begin{eqnarray*}
y_i^{k+1} &=& x_i^k+\alpha_k(x_i^k - x_i^{k-1}),\\
x_i^{k+1}&=& S_{\rho\lambda_k}\left(y_i^{k+1}-\lambda_k \nabla f(y^{k+1})_i\right)=
\left\{
     \begin{array}{lr}
       z_i^{k+1}: -h_i^*z_i^{k+1}\geq 0 \\
       0  : \text{else}
     \end{array}
   \right.
\end{eqnarray*}
Equivalently, for $k>K$,
\begin{eqnarray*}
y_E^{k+1} &=& x_E^k+\alpha_k(x_E^k - x_E^{k-1}),\\
x_E^{k+1}&=& P_{O_E}(y_E^{k+1}-\lambda_k \nabla \phi(y_E^{k+1})).
\end{eqnarray*}

Due to Theorem \ref{cor:forchamb}, the same arguments hold for parameter choices such that $\sum_k\|\Delta_k\|^2$ is finite, $\|x^k-x^*\|$ is bounded for all $k$ and some $x^*\in X^*$, and $\alpha_k$ is bounded. However there is no explicit upper bound on $K$. 

\end{proof}

In principle one could switch to minimizing $\phi$ directly once the algorithm has reduced to (\ref{eq:proj_HB1})-(\ref{eq:proj_HB2}). This would allow for a larger step-size, since the Lipschitz constant of $\nabla\phi$ is less than $L$. However it is not possible to know with certainty that the algorithm has transitioned to the form (\ref{eq:proj_HB1})-(\ref{eq:fista_mo3}) unless the number of iterations exceeds the upper bound $\max\{\overline{K}_D,\overline{K}_E\}$, although we discuss some heuristics for identifying this transition in Section \ref{sec:underD}. The main drawback of this strategy is that once it switches to minimizing $\phi$ directly the support of $x^k$ is fixed. Therefore any mismatch between $\supp(x^k)$ and $\supp(x^*)$ is not identified and the algorithm will not necessarily converge to an optimal point. In the next section we discuss a method which uses the optimal momentum for minimizing $\phi$ yet continues to use a smaller step-size and is therefore guaranteed to converge to a minimizer. 


\subsection{A Simple Locally Optimal Parameter Choice for I-FBS}
\label{sec:paramsfist}
The analysis of the previous three sections shows that, after a finite number of iterations, I-FBS (subject to parameter conditions) reduces to minimizing the function $\phi$ subject to an orthant constraint. Even though $f$ is not assumed to be strongly convex, $\phi$ might be. If this function is strongly convex, the asymptotic rate of convergence can be determined by the worst-case condition number of the Hessian. Throughout this section let  Let $H_{EE}(v)$ be the Hessian of $\phi$ evaluated at $v$. In terms of strategies for choosing $\{\alpha_k\}_{k\in\mathbb{N}}$ and $\{\lambda_k\}_{k\in\mathbb{N}}$, one approach is to choose them to obtain the best iteration complexity for minimizing $\phi$. In the following Corollary, we provide a simple fixed choice which does this and thus optimizes the asymptotic iteration complexity.

\begin{corollary}
Suppose that Assumption SO holds, and $\phi$ is strongly convex. Let $x^*$ be the unique minimizer of Problem SO and $l_E$ be the strong convexity parameter of $\phi$. If $\lambda\in(0,1/L]$,
\begin{eqnarray}
\lambda_k=\lambda\hbox{ and } \alpha_k=\frac{1-\sqrt{l_E\lambda}}{1+\sqrt{l_E\lambda}}\quad\forall k\in\mathbb{Z}_+,
\label{eq:fistamo}
\end{eqnarray}
 then the iterates $\{x^k\}_{k\in\mathbb{N}}$ of I-FBS converge to $x^*$ linearly and $F(x^k)$ converges to $F^*$ linearly. Indeed 
 \begin{eqnarray}
 \label{eq:fistal1ls_rate}
 F\lbr x^k\rbr-F^*
 = 
 O
 \left( 
 \left(1-\sqrt{l_E\lambda}\right)^{k}
 \right).
 \end{eqnarray}
 \label{cor:FISTAforL1LS}
 \end{corollary}
 \begin{proof}
 The analysis of the previous sections shows that, for the given choice of $\{\alpha_k\}_{k\in\mathbb{N}}$ and $\{\lambda_k\}_{k\in\mathbb{N}}$ there exists a $K$ such that, for all $k>K$, (\ref{eq:proj_HB1}) and (\ref{eq:proj_HB2}) hold, and $F(x^k)=\phi(x_E^k)$. 
 Thus for $k>K$ the algorithm is equivalent to Nesterov's method (Section 2.2 \cite{nesterov2004introductory}) applied to the strongly convex function $\phi$ subject to the orthant constraint $O_E$. Therefore we apply Theorem 2.2.1 of \cite{nesterov2004introductory} with the fixed parameter choice (constant scheme 3, discussed on p. 76 of \cite{nesterov2004introductory}). The only difference compared to Theorem 2.2.1 is that we allow for step-sizes other than $1/L_E$, where $L_E$ is the Lipschitz constant of $\nabla\phi$. Note that $L\geq L_E$. This minor change is discussed on p. 72. of \cite{nesterov2004introductory}. Setting $\lambda=1/L$ (the maximum allowed step-size) gives:
  \begin{eqnarray}
  \label{bnn}
  F\lbr x^k\rbr-F^*
  = 
  O
  \left( 
  \left(1-\sqrt{\frac{l_E}{L}}\right)^{k}
  \right).
  \end{eqnarray}
  Another minor issue to note is that the minimization is constrained to the simple convex set $O_E$. This does not effect the convergence of Nesterov's method, as discussed in Constant Step Scheme (2.2.17) of \cite{nesterov2004introductory}. 
  
  By the strong convexity of $\phi$, the sequence $\{x^k\}_{k\in\mathbb{N}}$ also achieves linear convergence with the same iteration complexity.
 \end{proof}
 
 The iteration complexity with this parameter choice is
  \begin{eqnarray}
 \Omega\lbr\sqrt{\frac{L}{l_E}}\log\lbr\frac{1}{\epsilon}\rbr\rbr. 
 \label{eq:itercomp}
  \end{eqnarray}
 This is the best asymptotic iteration complexity that can be achieved by I-FBS using this step-size \cite{nesterov2004introductory}. Therefore we will refer to it as the locally optimal choice. Indeed it is better than the iteration complexity of ISTA \cite{Hale:2008} (which corresponds to I-FBS with $\alpha_k$ equal to $0$) which is
  $$\Omega\lbr\frac{L}{l_E}\log\lbr\frac{1}{\epsilon}\rbr\rbr.$$
  We will see in the next section that (\ref{eq:itercomp}) is better than the iteration complexity achieved by the FISTA-like choices of \cite{Beck:2009}, \cite{tseng2008accelerated} and \cite{chambolle2014weak}. 
  
  In practice the optimal momentum is not known a priori as it depends on the smallest eigenvalue of $H_{EE}$. The momentum could be estimated periodically based on the smallest eigenvalue of the Hessian corresponding to the current support set, or adapted based on the behavior of $F(x^k)$ (see Section \ref{sec:adamo} and Section \ref{sec:sims}).
  
  

The authors of \cite{o2012adaptive} proposed an adaptive momentum restart scheme for Nesterov's method in the case of smooth optimization (i.e. $g(x)=0$). Corollary \ref{cor:q_lin} implies that the scheme can also be used for I-FBS applied to Problem SO. This follows because I-FBS with parameter choice (\ref{eq:fista_mo2}) reduces to minimizing a smooth function $\phi$ after a finite number of iterations, after which the momentum restart scheme can be used. Referring to the analysis of \cite{o2012adaptive}, it can be shown that the method will have the same iteration complexity as given in (\ref{eq:itercomp}).

Local linear convergence can also be proved when the local function $\phi$ is not strongly convex, but the limit point of the iterations obeys the ``strict-complementarity" condition: $E=\supp(x^*)$. Furthermore the Hessian matrix of $\phi$ must be invariant in a region containing the limit. In the following corollary, let $x^* = \lim_{k\to\infty}x^k$, which is in $X^*$ by Theorem \ref{thm:fista_lyap}.
\begin{corollary}
Suppose Assumption SO holds and $E=\supp(x^*)$, where $\lim_{k\to\infty} x^k=x^*\in X^*$. Let $H_{EE}(x)$ be the Hessian of the function $\phi$ defined in (\ref{eq:phidef}). Let $\hat{l}_E$ be the smallest non-zero eigenvalue of $H_{EE}(x_E^*)$. Assume the range space of $H_{EE}$ is invariant in some neighborhood $N^*$ around $x^*$. If all eigenvalues of $H_{EE}(x_E^*)$ are zero, $x^k=x^*$ after a finite number of iterations, for any choice of $\{\lambda_k,\alpha_k\}$ satisfying the conditions of Theorem \ref{thm:finite} or Theorem \ref{cor:forchamb}. If $\hat{l}_E>0$, $\lambda\in(0,1/L]$, 
\label{cor:strict_comp}
\begin{eqnarray}
\lambda_k=\lambda\hbox{ and } \alpha_k=\frac{1-\sqrt{\hat{l}_E\lambda}}{1+\sqrt{\hat{l}_E\lambda}}\quad\forall k\in\mathbb{Z}_+,
\label{eq:fistamo2}
\end{eqnarray}
then the iterates $x^k$ of I-FBS converge to $x^*$ linearly and $F(x^k)$ converges to $F^*$ linearly. Indeed 
 \begin{eqnarray*}
 \label{eq:fistal1ls_rate2}
 F\lbr x^k\rbr-F^*
 = 
 O
 \left( 
 \left(1-\sqrt{\hat{l}_E\lambda}\right)^{k}
 \right).
 \end{eqnarray*}
\end{corollary}
 \begin{proof}
 
The proof proceeds almost identically to Theorem 4.11 of \cite{Hale:2008}. Note that if $x^k\to x^*$ than $y^k\to x^*$. Now Lemma 5.3 of \cite{Hale:2008} can be directly applied to I-FBS to say that, after a finite number of iterations,
\begin{eqnarray}
x_i^{k+1} = y_i^{k+1} - \lambda_k(\nabla f(y^{k+1})_i - h_i^*)\quad\forall i\in \supp(x^*).
\label{ffm}
\end{eqnarray}
Assume $k$ is large enough that (\ref{ffm}) holds, $x^k\in N^*$, and $k>\max\{\overline{K}_D,\overline{K}_E\}$. Since $x_i^k$ is $0$ for all $i\in D$, it suffices to consider $i\in E=\supp(x^*)$. Recall that $H(x)$ is the Hessian of $f$ evaluated at $x$. Now, let $\overline{H}^k$ be defined as
\begin{eqnarray*}
\overline{H}^k\triangleq\int_0^1 H(x^*+t(x^k-x^*))dt
\end{eqnarray*}
By assumption the range spaces of $\overline{H}^k$ are now invariant over $k$. For a matrix $W$, let $W_{EE}$ be the $|E|\times|E|$ submatrix of $W$ with row and column indices given by $E$. Let $P$ be the orthogonal projection onto the range space of $H_{EE}$. Since $E=\supp(x^*)$, equation (\ref{ffm}) can be used to claim that
\begin{eqnarray}
x_E^{k+1}
= y_E^{k+1} - \lambda(\nabla f(y^{k+1}) - h^*)_E
= y_E^{k+1} - \lambda\overline{H}^k_{EE}(y_E^{k+1}-x_E^*).
\end{eqnarray}
This follows from the mean value theorem. At each iteration the term $- \lambda\overline{H}^k_{EE}(y_E^{k+1}-x_E^*)$ stays in the range space of $H_{EE}$, which implies that the null-space components of the iterates have already converged.  In other words, for $k$ sufficiently large,
$$
(I-P)(x_E^k-x_E^*)=0.
$$
If $\hat{l}_E>0$, it suffices to consider the convergence of $\{Px^k\}$, that is, consider the component in the range space of $H_{EE}$. Since $\phi$ restricted to the range space of $H_{EE}$ is strongly convex, we now simply repeat the arguments of Corollary \ref{cor:q_lin} and the result follows.
 \end{proof}

\subsection{``Underdamped" I-FBS}
\label{sec:underD}
In \cite{o2012adaptive}, the behavior of the FISTA-like methods when applied to strongly convex functions was investigated. It was shown that for such functions if $\alpha_k\approx 1$, the algorithm moves into what is known as an ``underdamped regime", which leads to oscillations in the objective function at a predictable frequency, and a sub-optimal iteration complexity. The results of the preceding sections allow us to extend this analysis to Problem SO, despite it being nonsmooth and not strictly convex. 

The analysis of \cite{o2012adaptive} revealed that if $\alpha_k\approx 1$, the trace of the objective function values will oscillate with a frequency proportional to $1/\sqrt{\kappa}$ where $\kappa$ is the condition number of the Hessian at the minimum. The iteration complexity in the high-momentum regime with step-size $\lambda=1/L$ is
$$
\Omega\lbr
\kappa\log \lbr\frac{1}{\epsilon}\rbr
\rbr
$$ 
More precisely, the behavior $F(x^k)\approx C (1-\frac{1}{\kappa})^k \cos^2(k/\sqrt{\kappa})$ is observed. Now Corollary \ref{cor:q_lin} shows that after a finite number of iterations, FISTA (with parametric constraints) reduces to minimizing $\phi$ (defined in (\ref{eq:phidef})) subject to an orthant constraint. Therefore we can apply the analysis of \cite{o2012adaptive} once the algorithm is in this regime. Thus if $l_E>0$, FISTA obeys the conditions of Theorem \ref{thm:finite} or \ref{cor:forchamb}, and $\alpha_k\to 1$, then the iteration complexity will be \cite{o2012adaptive}
$$
\Omega\lbr
\frac{L}{l_E}\log \lbr\frac{1}{\epsilon}\rbr
\rbr,
$$
which is worse than the iteration complexity achieved with the locally optimal choice given in Corollary \ref{cor:FISTAforL1LS}. The trace of the objective function will exhibit oscillations with period $\sqrt{l_E/L}$. This result directly applies to the FISTA-like choices of \cite{Beck:2009}, \cite{tseng2008accelerated} and \cite{chambolle2014weak}. The result also applies when $l_E=0$ under the strict-complementarity condition. In this case replace $l_E$ with $\hat{l}_E$ defined in Corollary \ref{cor:strict_comp}.

\subsection{An Adaptive Modification}
\label{sec:adamo}
In numerical experiments we have noticed that it can take many iterations for the optimal manifold to be identified by I-FBS. This means that one of the FISTA-like choices can outperform I-FBS with our locally optimal choice (\ref{eq:fistamo}) before the optimal manifold is identified. This is because the FISTA-like choices guarantee $O(1/k^2)$ convergence during this phase whereas the locally optimal choice does not have a guaranteed rate until the optimal manifold is identified (although an $O(1/k)$ rate can probably be established following the analysis of \cite{chambolle2014weak}, however this is beyond the scope this paper). On the other hand the analysis of the previous section showed that the FISTA-like choices have poor performance once the algorithm is in the optimal manifold. In summary the FISTA-like choices have excellent global properties but poor local properties. 

In light of this we propose the following adaptive heuristic. We use the condition, $F(x^k)>F(x^{k-1})$, as an indication the algorithm is at least approximately operating in the optimal manifold. This is because with a FISTA-like choice the algorithm will eventually converge to the optimal manifold and then the function values will start to oscillate. So the adaptive modification is the following. Use Beck and Teboulle's parameter choice of (\ref{eq:fista_mo2}) until $F(x^{k})>F(x^{k-1})$. For all iterations after, use the locally optimal momentum given in (\ref{eq:fistamo}). We call this scheme FISTA-AdOPT. See Experiment 1 for empirical results. It is worth mentioning that it is better to use the condition $(y^{k+1}-x^{k+1})^T(x^{k+1}-x^k)>0$ rather than $F(x^k)>F(x^{k-1})$. It was shown in \cite{o2012adaptive} that the two conditions are equivalent however the first avoids computation of $F$.

The iterates of FISTA-AdOPT are guaranteed to converge to a solution and, until $F(x^k)>F(x^{k-1})$ occurs, the convergence in the objective function is $O(1/k^2)$. Furthermore, the asymptotic convergence rate is optimal and given by (\ref{bnn}). The main drawback of the scheme is that it might switch to the locally optimal parameter choice before the iterates are confined to the optimal manifold. Another drawback is that the locally optimal momentum parameter must be estimated, which involves computing the largest and smallest singular values of $A$ confined to the current support set. This computation is nontrivial when the support set is large. A better alternative is to use the adaptive restart scheme proposed in \cite{o2012adaptive}. Thanks to our analysis, this method is guaranteed to achieve the same asymptotic iteration complexity as the locally optimal parameter choice, but will also achieve the $O(1/k^2)$ performance in the transient regime prior to manifold identification (See the remarks after Corollary \ref{cor:FISTAforL1LS}). The practical performance of the momentum restart scheme and FISTA-AdOPT are compared in the next section.  

\subsection{The Splitting Inertial Proximal Method}

In \cite{MoudafiOliny}, Moudafi and Oliny introduced the Splitting Inertial Proximal Method (SIPM):

\begin{eqnarray}
x^{k+1} 
\in 
-\lambda_k\partial_{\epsilon_{k}}g(x^{k+1})+x^{k}-\lambda_k\nabla f(x^{k})+\alpha_k(x^k - x^{k-1}).
\label{eq:sipmdef}
\end{eqnarray} 
In fact they introduced it for the more general monotone inclusion problem. The method is a direct generalization of Polyak's heavy-ball with friction method to problems involving the sum of two functions. It differs from I-FBS in that the gradient w.r.t. $f$ is computed at $x^k$ rather than the extrapolated point $x^k + \alpha_k(x^k - x^{k-1})$. Our analysis of I-FBS in the case of sparse optimization can be extended easily to Moudafi and Oliny's method under the condition that $\sum_k\|\Delta_k\|^2$ is finite, for which sufficient conditions were established in \cite{MoudafiOliny}. 
\begin{theorem}
Suppose that Assumption SO holds. Assume $\epsilon_k$ is $0$ for all $k\in\mathbb{N}$, $0<\lambda_k< 2/L$ for all $k$, and there exists $\overline{\alpha}\geq0$ such that $0\leq\alpha_k\leq \overline{\alpha}$ for all $k$. If
$
\sum_k\|\Delta_k\|^2<\infty, 
$
and $\|x^k-x^*\|$ is bounded for all $x^*\in X^*$,
then there exists a constant $K>0$ such that, for all $k>K$ the iterates of Algorithm (\ref{eq:sipmdef}) applied to Problem SO satisfy
\begin{eqnarray*}
\label{thm:sipmfinite:eq2}
\sgn\left(x_i^{k}-\lambda_k\nabla f(x^{k})_i + \alpha_k\left(x_i^k-x_i^{k-1}\right)\right)
&=&-\frac{h^*_i}{\rho},\ \forall i \in E,
\end{eqnarray*}
and
\begin{eqnarray*}
x^k_i = 0,\ \forall i \in D.
\label{thm:sipmfinite:eq1}
\end{eqnarray*}
\end{theorem}
\begin{proof}
The proof follows in a similar way to Theorems \ref{thm:finite} and \ref{cor:forchamb}. Equation (\ref{afs}) is proved by following similar arguments as in the proof of Theorem \ref{thm:finite}. We include the salient differences.

Recall that $\nu_k=\rho\lambda_k$. If
\begin{eqnarray}
\sgn(x_i^k-\lambda_k\nabla f(x^k)_i+\alpha_k(x^k_i-x^{k-1}_i))
&\neq&
\sgn(x_i^*-\lambda_kh_i^*)
=-h_i^*/\rho
\nonumber\\&&
\hbox{for some}\ i\in E,
\label{incorrect_sgn}
\end{eqnarray}
then,
\begin{eqnarray}
\nonumber
|x_i^{k+1}-x_i^*|^2
&=&
\left|
S_{\nu_k}
\left(x_i^k-\lambda_k\nabla f(x^k)_i+\alpha_k\left(x^k_i-x^{k-1}_i\right)\right)
-S_{\nu_k} (x_i^*-\lambda_k h_i^*)
\right|^2
\\
&\leq&
\left|
x_i^k-\lambda_k\nabla f(x^k)_i+\alpha_k\left(x^k_i-x^{k-1}_i\right)
-(x_i^*-\lambda_k h_i^*)
\right|^2
-\nu_k^2
\label{aas}
\end{eqnarray}
where (\ref{aas}) follows for the same reasons as given for (\ref{eq:mid}). Continuing to follow the proof of Theorem \ref{thm:finite}, we say the following: if (\ref{incorrect_sgn}) holds, then
\begin{eqnarray}
\nonumber
\|x^{k+1}-x^*\|^2
&=&
\sum_{j\neq i}|x^{k+1}_j-x^*|^2 +|x_i^{k+1}-x^*|^2
\\
&\leq&
\sum_{j\neq i}
\left|
x_j^k-\lambda_k\nabla f(x^k)_j+\alpha_k\left(x^k_j-x^{k-1}_j\right)
-(x_j^*-\lambda_k h_j^*)
\right|^2
\nonumber\\&&
+
\left|
x_i^k-\lambda_k\nabla f(x^k)_i+\alpha_k\left(x^k_i-x^{k-1}_i\right)
-(x_i^*-\lambda_k h_i^*)
\right|^2
\nonumber\\&&
-\nu_k^2
\label{asfjg}\\\nonumber
&\leq&
\|x^k-\lambda_k\nabla f(x^k)+\alpha_k\Delta_k
-(x^*-\lambda_k h^*)\|^2
-\nu_1^2
\\\label{opi}
&\leq&
\|x^k-x^*\|^2
+\overline{\alpha}^2\|\Delta_k\|^2
+2\overline{\alpha}\|x^k-x^*\|\|\Delta_k\|
 \end{eqnarray}
Equation (\ref{asfjg}) follows from the element-wise nonexpansiveness of $S_{\nu}$ and (\ref{aas}), and (\ref{opi}) follows from the nonexpansiveness of $I-\lambda\nabla f$, by application of Cauchy-Schwartz, and by substituting the upper bound $\overline{\alpha}$ for $\alpha_k$. Equation (\ref{opi}) is identical to (\ref{afs}).
From this point on, the proof is identical to Theorem \ref{cor:forchamb}. As before we cannot explicitly bound the number of iterations unless we know an upper bound for $\sum_k \Delta_k^2$.
\end{proof}

Moudafi and Oliny provided a condition on the parameters under which $$\sum_k\|\Delta_k\|^2<\infty,$$ and $\|x^k-x^*\|$ is bounded for all $x^*\in X^*$. Choose the sequence $\{\alpha_k\}_{k\in\mathbb{N}}$ to be non-decreasing, the constant $\overline{\alpha}<1/3$, and the sequence $\{\lambda_k\}_{k\in\mathbb{N}}$ to be bounded away from $2/L$. They showed that this implies $\|x^k-x^*\|$ is bounded for all $k$ and $x^*\in X^*$ and $$\sum_k\alpha_k\|\Delta_k\|^2<\infty.$$ Now using the fact that $\alpha_k$ is less than $1/3$, $\sum_k\|\Delta_k\|^2<\infty$. If $\alpha_k$ is zero for all $k$ the theorem follows from \cite{Hale:2008} since the algorithm reduces to ISTA.

We have proved that finite convergence in sign on $E$ and to $0$ on $D$ holds for SIPM applied to Problem SO. An analogous result to Corollary \ref{cor:q_lin} can now be shown. After a finite number of iterations, SIPM reduces to HBF projected onto a quadrant. Because of its similarity to Corollary \ref{cor:q_lin}, the proof is omitted. 

\begin{corollary}
Assume $0<\lambda_k< 2/L$ for all $k$, and there exists $\overline{\alpha}\geq0$ such that $0\leq\alpha_k\leq \overline{\alpha}$ for all $k$. If $\sum_k\|\Delta_k\|^2<\infty$, and $\|x^k-x^*\|$ is bounded for all $x^*\in X^*$, then there exists a $K>0$ such that for $k>K$, the iterates of SIPM satisfy
\begin{eqnarray*}
x_E^{k+1} &=& P_{O_E}\left(x^k_E-\lambda_k\nabla\phi(x_E^k)+\alpha_k\left(x_E^k-x_E^{k-1}\right) \right) 
\\
x_D^{k}&=& 0
\\
F(x^k) &=& \phi(x_E^k) 
\end{eqnarray*}
where $\phi$ and $O_E$ are defined in (\ref{eq:phidef}) and (\ref{O_E_def}) respectively. 
\end{corollary}

Since SIPM reduces to projected HBF, parameter choices could be made to optimize the performance of HBF on the optimal manifold. Such computations have been carried out elsewhere \cite{johnstone15} for the special case of Problem $\ell_1$-LS. The analysis closely follows the original work by Polyak for HBF \cite{polyak1964some}. For strongly convex quadratic functions, HBF obtains linear convergence with a $\sqrt{\kappa}$-times lower rate than gradient descent, where $\kappa$ is the condition number of the Hessian. For Problem $\ell_1$-LS, when $l_E>0$ the optimal asymptotic iteration complexity of HBF turns out to be equal to that of I-FBS with our optimal choice given in (\ref{eq:itercomp}). However we can only guarantee local linear convergence for non-decreasing choices of $\alpha_k$ in the range $[0,\overline{\alpha}]$ with $\overline{\alpha}$ less than $1/3$. So the locally optimal value of the momentum must be less than $1/3$ for this to hold. Otherwise convergence is linear with a worse iteration complexity \cite{johnstone15}.

\section{Numerical Simulations}
\label{sec:sims}
 We now compare several choices of parameters for I-FBS applied to a random instance of Problem $\ell_1$-LS. To compute $E$ and thus find the locally optimal parameter choice given in (\ref{eq:fistamo}), we use the interior point solver of \cite{kim2007interior} to find a solution to a target duality gap of $10^{-8}$. We then compute $h^*=\nabla f(x^*)$, and approximate the set $E$ by the set of all entries such that $\rho-|h^*_i|$ is smaller than $10^{-4}$. We also use the interior point solver to find an estimate of $F^*$. Recall that $l_E$ denotes the smallest eigenvalue of $A_E^T A_E$ and note that $l_E$ was greater than $0$ in all experiments we ran. The parameter choices under consideration will be referred to by the following designations. 
\begin{itemize}
\item {\bf I-FBS-OPT.} I-FBS with our locally optimal parameters derived in Section \ref{sec:paramsfist}, given in (\ref{eq:fistamo}), with $\lambda=1/L$. Recall that this choice optimizes the local convergence rate. Note that this is not a practically implementable algorithm as it depends on the optimal momentum. We include it to verify the theory of Section \ref{sec:l1ls}.
\item {\bf ISTA-OPT \cite{Hale:2008}.} ISTA with parameters chosen to optimize the local convergence rate. Corresponds to I-FBS with $\alpha_k=0$ and  step-size $\lambda_k=2/(L+l_E)$. As with I-FBS-OPT this step-size cannot be used in practice as it depends on $l_E$. Using $\lambda_k=1/L$ is common in practice but gives worse convergence rate. 
\item {\bf FISTA-BT \cite{Beck:2009}}. FISTA with Beck and Teboulle's parameters given in (\ref{eq:fista_mo2}).
\item {\bf FISTA-AdOPT.} Our proposed method, see Section \ref{sec:adamo}. FISTA with Beck and Teboulle's choice until $F(x^k)>F(x^{k-1})$, (equivalently  $(y^{k+1}-x^{k+1})^T(x^{k+1}-x^k)>0$), I-FBS-OPT for all iterations thereafter. We estimate the locally optimal momentum by using $\supp(x^k)$ as a surrogate for $E$, since $\supp(x^k)\subset E$ for $k$ large enough. We compute the smallest eigenvalue of $A_{\supp(x^k)}^TA_{\supp(x^k)}$ one time when $F(x^k)>F(x^{k-1})$ and compute the locally optimal momentum using (\ref{eq:fistamo}). We use this fixed value of the momentum from then on. This algorithm is practically implementable so long as $|\supp(x^k)|$ is not too large. 
\item {\bf FISTA-AdRe \cite{o2012adaptive}.} Adaptive momentum restart for FISTA. Beck and Teboulle's parameter choice given in (\ref{eq:fista_mo2}) except set $t_k$ to $0$ whenever $F(x^k)>F(x^{k-1})$ (equivalently  $(y^{k+1}-x^{k+1})^T(x^{k+1}-x^k)>0$). Our analysis shows this method achieves the optimal asymptotic iteration complexity derived in Corollary \ref{cor:q_lin} (See the remarks after Corollary \ref{cor:FISTAforL1LS}).
\end{itemize}
All algorithms are initialized to $x^1=x^{0}=0$.

{\bf Experiment Details.} We create a random instance of Problem $\ell_1$-LS, with $A$ of size $300\times 2000$ and with entries drawn i.i.d. from $\mathcal{N}(0,0.01)$. The vector $b$ is $Ax_0$ with $x_0$ being $50$-sparse having non-zero entries drawn i.i.d. $\mathcal{N}(0,1)$. The regularization parameter $\rho$ is set to $1$. 

The results are shown in Fig. \ref{fig:Exp4}. Both FISTA-AdOPT and FISTA-AdRe inherit the $O(1/k^2)$ convergence rate of FISTA-BT during the transient period. However they also have the optimal asymptotic rate of I-FBS-OPT. FISTA-BT begins to oscillate once the optimal manifold is identified and the iteration complexity is worse than I-FBS-OPT, as predicted in Section \ref{sec:underD}. The performances of FISTA-AdOPT and FISTA-AdRe are similar. FISTA-AdOPT has the advantage that it will not oscillate like FISTA-AdRe, which has to be continually reset once the momentum exceeds the optimal value. However FISTA-AdOPT requires computation of an estimate of the locally optimal momentum which depends on smallest eigenvalue of $A$ restricted to the support. 

\begin{figure}[!h]
\centering
\includegraphics[width=4in]{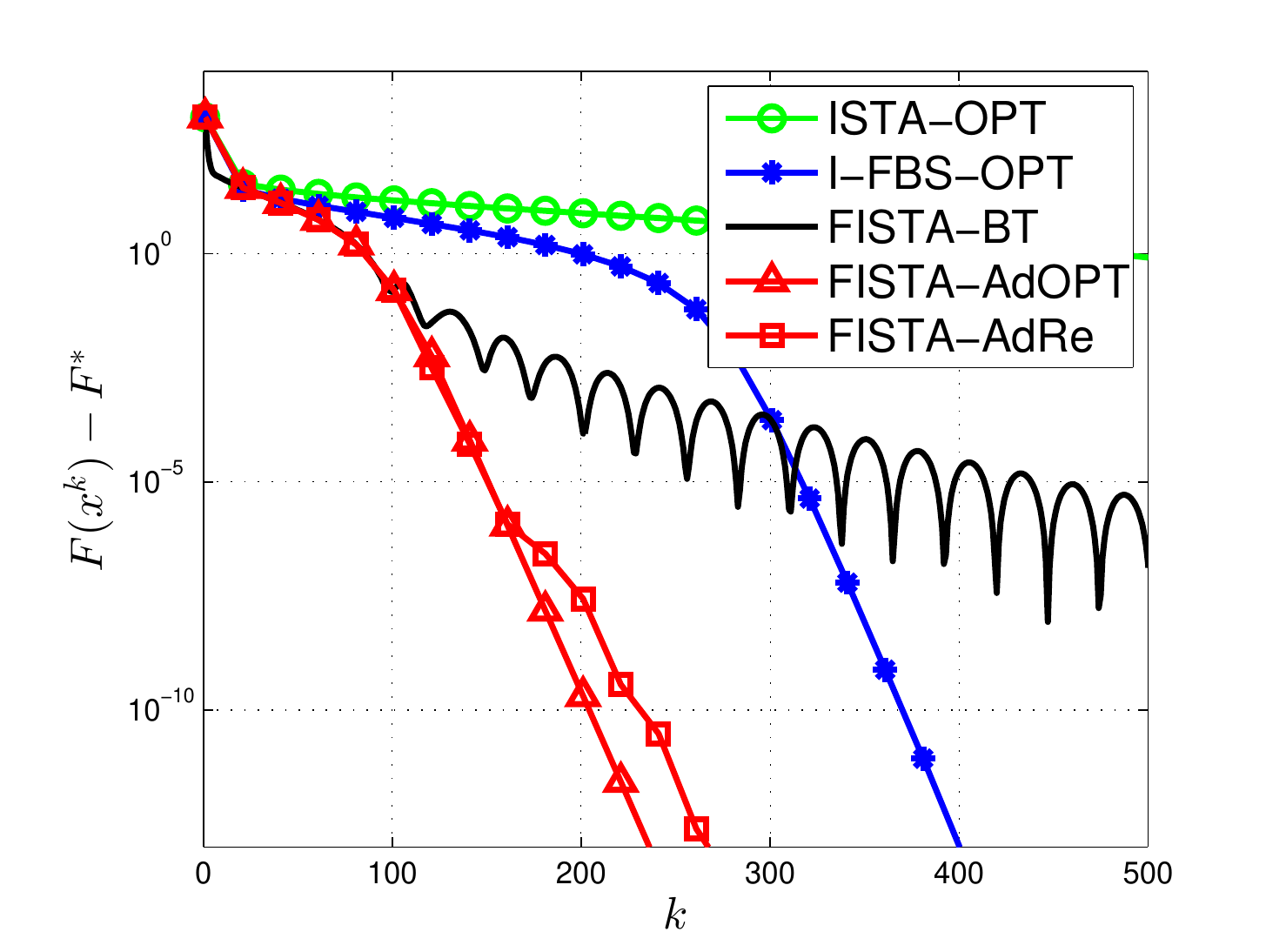}
\caption{Experiment results: showing $F(x^k)-F^*$ versus iteration $k$ for several algorithms.}
\label{fig:Exp4}
\end{figure}

\section{Conclusions}
In this paper, we applied a Lyapunov analysis to a family of inertial forward-backward splitting methods for convex composite minimization. We have proved weak convergence under the following broad conditions with the standard assumptions on the objective function: for the momentum parameter, $0\leq\alpha_k\leq 1$ for all $k$ and $\lim\sup \alpha_k<1$, and for the step-size, non-decreasing and $0<\lambda_k\leq 1/L$. These conditions are more general than the specific sequences studied in \cite{Beck:2009}, \cite{tseng2008accelerated} and \cite{chambolle2014weak} and less restrictive than the conditions derived in \cite{lorenz2014inertial}. We considered in detail the behavior of I-FBS applied to sparse optimization problems. With the aid of some results from the Lyapunov analysis we were able to show that I-FBS achieves local linear convergence for these problems, with finite convergence of certain quantities. The local linear convergence results also apply to the FISTA-like methods of \cite{Beck:2009}, \cite{tseng2008accelerated} and \cite{chambolle2014weak}.  

An interesting direction of future research is to see if this local linear convergence behavior holds for a more general class of problems satisfying certain properties such as partial smoothness and local strong convexity, as considered in \cite{liang2014local} for the forward-backward algorithm. 

{\bf Acknowledgments.} We would like to thank Olgica Milenkovic, Angelia Nedi\'{c}, Amin Emad, Antonin Chambolle and Charles Dossal for helpful discussions. 




\bibliographystyle{IEEEbib}
\bibliography{refs}

\end{document}